\documentclass[12pt]{amsart}
\usepackage[inner=1in, outer=1in, top=1in, bottom=1in]{geometry}

\usepackage[utf8]{inputenc}
\usepackage{amsmath}
\usepackage{amsfonts}
\usepackage{algpseudocode}
\usepackage{algorithm}

\usepackage{url, mathrsfs}

\usepackage[colorlinks]{hyperref}
\usepackage[capitalize, nameinlink]{cleveref}
\usepackage{amsthm}
\author{Jonathan Ni\~no-Cortes}
\author{Cynthia Vinzant}
\title{The convex algebraic geometry of higher-rank numerical ranges}
\newtheorem{theorem}{Theorem}[section]
\newtheorem{proposition}[theorem]{Proposition}
\newtheorem{lemma}[theorem]{Lemma}
\newtheorem{corollary}[theorem]{Corollary}

\theoremstyle{definition}
\newtheorem{example}[theorem]{Example}
\newtheorem{definition}[theorem]{Definition}
\newtheorem{remark}[theorem]{Remark}

\usepackage{graphicx}
\usepackage{caption}
\usepackage{subcaption}
\usepackage{bbm}

\newcommand{\R}{\mathbb{R}}
\newcommand{\Pj}{\mathbb{P}}
\newcommand{\C}{\mathbb{C}}
\newcommand{\PP}{\mathbb{P}}
\newcommand{\W}{\mathcal{W}}
\newcommand{\ii}{\mathbbm{i}}
\newcommand{\V}{\mathcal{V}}

\DeclareMathOperator{\conv}{conv}
\DeclareMathOperator{\cone}{conicalHull}

\DeclareMathOperator{\Sing}{Sing}

\DeclareMathOperator{\Span}{span}

\usepackage[draft]{fixme}
\fxsetup{theme=color, mode=multiuser, noinline}
\FXRegisterAuthor{j}{J}{\color{red}Jonathan}
\FXRegisterAuthor{c}{C}{\color{blue}Cynthia}

\begin{document}
	\algnewcommand\algorithmicswitch{\textbf{switch}} \algnewcommand\algorithmiccase{\textbf{case}}
	\algdef{SE}[SWITCH]{Switch}{EndSwitch}[1]{\algorithmicswitch\ #1\ \algorithmicdo}{\algorithmicend\ \algorithmicswitch}
	\algdef{SE}[CASE]{Case}{EndCase}[1]{\algorithmiccase\ #1}{\algorithmicend\ \algorithmiccase}
	\algtext*{EndCase} \algtext*{EndSwitch} \algtext*{EndIf} \algtext*{EndFor}

	\begin{abstract}
		The higher-rank numerical range is a convex compact set generalizing the classical
		numerical range of a square complex matrix, first appearing in the study of
		quantum error correction. We will discuss some of the real algebraic and convex
		geometry of these sets, including a generalization of Kippenhahn’s theorem, and
		describe an algorithm to explicitly calculate the higher-rank numerical range
		of a given matrix.
	\end{abstract}

	\maketitle

	\section{Introduction}

	The \emph{numerical range} of a complex matrix $A \in \mathbb{C}^{n\times n}$
	is given by
	\[
		\W(A) = \{x^{*}Ax : x \in \mathbb{C}^{n} \text{ and }x^{*}x = 1 \}.
	\]
	It is a classical theorem by Toeplitz and Haussdorff
	\cite{Hausdorff1919, Toeplitz1918} that $\W(A)$ is a convex compact subset of
	the complex plane. See \cite[Section II.14]{BarvinokAlexander2015CiC} for a modern
	proof. In 1951, Kippenhahn showed that $\W(A)$ is the convex hull of a real algebraic
	curve. This curve is dual to the algebraic curve defined by the vanishing of
	the \emph{Kippenhahn polynomial} of $A$, namely
	\[
		f_{A} = \det\left(tI_{n} + x \Re(A) + y\Im(A)\right)
	\]
	where $I_{n}$ denote the $n\times n$ identity matrix, $\Re(A) = \frac{1}{2}(A+A
	^{*})$ and $\Im(A) = \frac{1}{2\ii}(A-A^{*})$.

	Motivated by problems in compression and quantum error correction, in 2006, Choi,
	Kribs, \.Zyczkowski introduced a generalization called the higher-rank
	numerical range of a matrix \cite{CHOI2006828}, \cite{CHOI200677}. Given a
	matrix $A\in \C^{n\times n}$ and a positive integer $1\leq k \leq n$, the \emph{rank-$k$}
	numerical range of $A$ is defined to be
	\begin{equation}
		\label{eq:projections}\Lambda_{k}(A) = \left\{ \mu \in \mathbb{C}: \exists \text{
		a rank-}k \text{ orthogonal projection $P$ s.t.~}PAP = \mu P \right\}.
	\end{equation}
	In the context of quantum error correcting, elements $\mu\in \Lambda_{k}(A)$ are
	known as compression-values \cite{CHOI200677}. An equivalent description of
	this set is
	\begin{equation}
		\Lambda_{k}(A) = \left\{ \mu \in \mathbb{C}: \exists X \in \mathbb{C}^{n
		\times k}\text{ such that }X^{*}X = I_{k} \text{ and }X^{*}AX = \mu I_{k} \label{eq:higher_rank}
		\right\}.
	\end{equation}
	From this description, one can see that the higher-rank numerical ranges are nested
	and that the rank-one numerical range coincides with $\W(A)$. That is, $\W(A) =
	\Lambda_{1}(A)$ and $\Lambda_{k}(A) \supseteq \Lambda_{k+1}(A)$ for all $k$.
	They are also guaranteed to be non-empty for $k < n/3 + 1$
	\cite{ChiKwongNonEmptyness}.

	Choi et.~al.~conjectured that higher-rank numerical ranges are always convex, which
	was proved by Woerdeman in 2008. \cite{Woerdeman2008}. Independently, Li and Sze
	also proved that $\Lambda_{k}(A)$ is convex \cite{Li2008} by proving a description
	as an intersection of halfplanes, namely
	\begin{equation}
		\label{eq:half_plane}\Lambda_{k}(A) = \bigcap_{\theta \in [0,2\pi)}\left\{ \mu
		\in \mathbb{C}:{\rm Re}(e^{-i\theta}\mu) \leq \lambda_{k}(\Re(e^{-i\theta}A))
		\right\}
	\end{equation}
	where $\lambda_{1}(M)\geq \hdots \geq \lambda_{n}(M)$ denote the real
	eigenvalues of a Hermitian matrix $M$.

	\begin{example}
		\label{ex:quartic1} The $4\times 4$ weighted cyclic shift matrix
		\begin{align*}
			A = \begin{pmatrix}0&2 & 0 & 0 \\ 0 & 0 & 4 & 0 \\ 0 & 0 & 0 & 6 \\ 8 & 0 & 0 & 0\
\end{pmatrix} \ \text{ gives }\  f_{A} & =\det\begin{pmatrix}t & x-\ii y & 0 & 4 x+4 \ii y \\ \! x+\ii y \!& t &\! 2 x-2 \ii y \!& 0 \\ 0 &\! 2 x+2 \ii y \!& t &\! 3 x-3 \ii y \!\\ \!4 x-4 \ii y \!& 0 &\!3 x+3 \ii y \!& t\end{pmatrix} \\
			                                                                                                                          & = 25 (x^{4}+y^{4})+434 x^{2} y^{2}-30 (x^{2}+y^{2})^{2}+t^{4}.
		\end{align*}
		Figure~\ref{fig:quartic1} shows the zero set of $f_{A}(1,x,y)$ together with
		rank-one and rank-two numerical ranges of $A$. The dual curve of $f_{A}$ is
		defined by the polynomial
		\begin{align*}
			 & g_{A} = 15625a^{12}+ 273750a^{10}b^{2} + 90375a^{8}b^{4} + 549236a^{6}b^{6} + 90375a^{4}b^{8} + 273750a^{2}b^{10}+ 15625b^{12} \\
			 & - 1368750a^{10}- 17139750a^{8}b^{2} + 44934900a^{6}b^{4}+ 44934900a^{4}b^{6} - 17139750a^{2}b^{8} - 1368750b^{10}              \\
			 & + 47610625a^{8} + 429249700a^{6}b^{2} - 1058169786a^{4}b^{4} + 429249700a^{2}b^{6} + 47610625b^{8}                             \\
			 & - 838188000a^{6} - 5975989920a^{4}b^{2} - 5975989920a^{2}b^{4} - 838188000b^{6} + 7621461600a^{4}                              \\
			 & + 39076977600a^{2}b^{2} + 7621461600b^{4}- 30526848000a^{2} - 30526848000b^{2} + 21083040000.
		\end{align*}
		By Kippenhahn's theorem, $\W(A)$ is the convex hull of $\{a+\ii b\in \C : g(a
		,b)=0\}$.
	\end{example}
	\begin{figure}
		\includegraphics[height=2in]{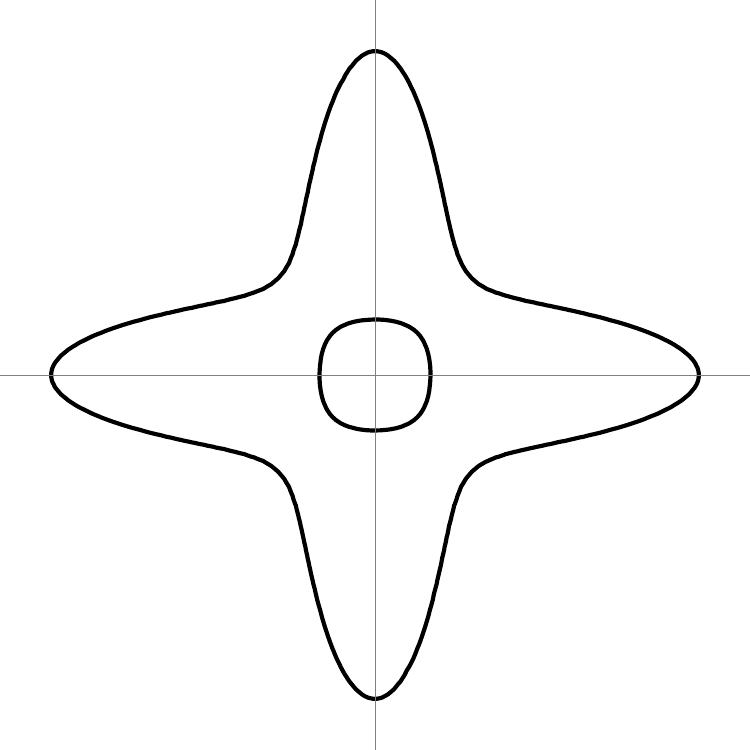}
		\includegraphics[height=2in]{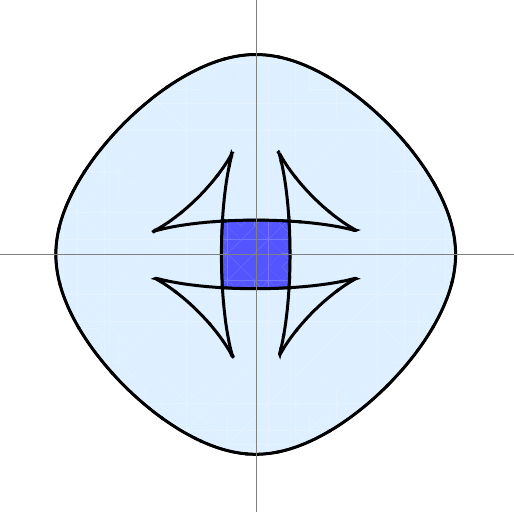}
		\caption{The curve $f_{A}=0$ and numerical ranges $\Lambda_{1}(A)$ and
		$\Lambda_{2}(A)$ of the $4\times 4$ matrix $A$ from Example~\ref{ex:quartic1}.
		The zero set of the polynomial $g_{A}$ vanishes on the boundaries of $\Lambda
		_{1}(A)$ and $\Lambda_{2}(A)$.}
		\label{fig:quartic1}
	\end{figure}

	Higher-rank numerical ranges arose in the study of quantum error-correcting codes.
	In this context, errors are represented as sets of operators that are applied to
	vectors in the Hilbert space of possible quantum states. The goal is to find a
	projection that makes the action of these operators on the data negligible, usually
	at the expense of restricting to a subspace, thereby lowering the capacity of
	the channel. Each projection is associated with a compression value, and the
	set of all possible compression values for rank-$k$ projections corresponds to
	the rank-$k$ numerical range. See \cite{CHOI2006828}, \cite{CHOI200677} for
	more details. Gau and Wu \cite{gau_higher_rank_2013} showed that that the Kippenhahn
	polynomial $f_{A}$ of a matrix $A$ completely determines all of its higher-rank
	numerical ranges and, conversely, any two matrices all of whose higher-rank numerical
	ranges coincide have the same Kippenhahn polynomial. Bebiano, Provid\'encia,
	and Spitkovsky \cite{BEBIANO2021246} give a description of the rank-$k$
	numerical range for certain families of matrices.

	Higher-rank numerical ranges have special behavior when $f_{A}$ is a power of
	a single linear form, i.e.~$f_{A} = (t+ax+by)^{n}$. In this case $\Lambda_{k}(A
	) = \{a+\ii b\}$ for all $k=1, \hdots, n$ \cite[Prop.~2.2]{CHOI2006828}.
	Throughout the paper we restrict ourselves to matrices $A$ for which $f_{A}$
	is not a power of a linear form.

	In this paper we take an algebraic approach to the problem of computing higher-rank
	numerical ranges. For $k=1$, Henrion~\cite{henrion2006} gives both an
	algebraic description of $\W(A) = \Lambda_{1}(A)$ and a description as the
	projection of a spectrahedron. Here we aim to generalize the algebraic part of
	this to $k>1$. Using the description \eqref{eq:half_plane}, we describe the
	convex dual of $\Lambda_{k}(A)$ in Section~\ref{sec:DualCone}. In Section~\ref{sec:MembershipTest},
	we provide an algorithm for testing membership of a fixed point in the rank-$k$
	numerical range of a matrix. In Section~\ref{sec:Dim01}, we explore the
	geometry of the curve $f_{A}=0$ when the rank-$k$ numerical range is nonempty
	but not full-dimensional. Building off of this characterization, we give an
	algorithm to compute $\Lambda_{k}(A)$ in Section~\ref{sec:algorithm}. We conclude
	by giving a gallery of interesting examples in Section~\ref{sec:gallery}.

	There are many other interesting generalizations of the numerical range. For
	example, Afshin et al. studied joint higher-rank numerical ranges
	\cite{Afshin2018}. Given an $m$-tuple $(A_{1}, \ldots A_{m})$ of Hermitian matrices
	the joint higher-rank numerical range is defined as
	\begin{equation}
		\begin{split}
			\Lambda_{k}(A_{1},\ldots A_{m}) := \{ \mu \in \mathbb{C}:&\: \exists X \in
			\mathbb{C}^{n \times k}\text{ s.t. }X^{*}X = I_{k} \\
			&\text{ and }X^{*}A_{j}X = \mu I_{k} \text{ for all }1 \leq j \leq m \label{eq:joint_higher_rank}
			\}.
		\end{split}
	\end{equation}
	Observe that when $m =2$, the previous set is the usual higher-rank numerical
	range. Specifically, $\Lambda_{k}(A_{1},A_{2}) = \Lambda_{k}(A)$, where $A = A_{1}
	+ \ii A_{2}$. This construction addresses the problem of having more than one
	error operator in the quantum channel. However, these sets are not convex in
	general, even for $k=1$ \cite{GUTKIN2004143}.

	{\bf Acknowledgements.} We would like to thank Pamela Gorkin, Didier Henrion, Ion
	Nechita, Mohab Safey El Din, and Rekha Thomas for helpful discussions and comments.
	Both authors received support from NSF grant DMS-2153746 and a Sloan Research
	Fellowship.

	\section{Preliminaries from convex algebraic geometry}

	\subsection{Notation}
	For a Hermitian matrix $M$, we use
	$\lambda_{1}(M) \geq \hdots \geq \lambda_{n}(M)$ to denote the eigenvalues of
	$M$. We often fix a matrix $A$ and use the notation
	\[
		\lambda_{k}(\theta) =\lambda_{k}(\Re(e^{-i\theta}A)) = \lambda_{k}(\cos(\theta
		) \Re A + \sin(\theta) \Im A) \ \ \text{ and }\ \
\lambda_{k}(x,y) = \lambda_{k}
		(x \Re A + y \Im A)
	\]
	where $\theta\in [0,2\pi)$ and $(x,y)\in \R$. Any $(x,y)\in \R^{2}$ can be
	written as $(x,y) = r(\cos(\theta),\sin(\theta))$ for some $r\geq 0$ and
	$\theta\in [0,2\pi)$, in which case $\lambda_{k}(x,y) = r\lambda_{k}(\theta).$

	\subsection{Convex geometry}
	A set $S\subset \mathbb{R}^{n}$ is \emph{convex} if for any points ${\bf x},{\bf y}
	\in S$, the line segment joining them, $\{\lambda{\bf x}+ (1-\lambda){\bf y}: 0
	\leq \lambda \leq 1\}$, is contained in $S$. We call $S$ a \emph{convex cone}
	if it is also closed under nonnegative scaling, or equivalently if for any two
	points ${\bf x},{\bf y}\in S$, the set of conic combinations
	$\{\lambda{\bf x}+ \mu{\bf y}: \lambda, \mu\in \mathbb{R}_{\geq 0}\}$ is
	contained in $S$. We call $F\subseteq S$ a \emph{face} of a convex set if for
	every ${\bf p}\in F$, if ${\bf p}= \lambda{\bf x}+ (1-\lambda){\bf y}$ for some
	$\lambda \in (0,1)$ where ${\bf x},{\bf y}\in S$, then ${\bf x},{\bf y}\in F$.
	That is, any way to express a point from $F$ as a convex combination of points
	from $S$ involve only points from $F$.

	The \emph{convex hull} of a set $S\subset \mathbb{R}^{n}$ is the smallest convex
	set containing $S$ and its \emph{conic hull} is the smallest convex cone
	containing $S$. We can write these as
	\begin{align*}
		{\rm conv}(S) & = \left\{\sum_{i=1}^{k} \lambda_{i} p_{i} : \ k\in \mathbb{N}, \ p_{i}\in S, \ \lambda_{i}\geq 0, \ \sum_{i=1}^{k}\lambda_{i} = 1 \right\}, \text{ and } \\
		\cone(S)      & = \left\{\sum_{i=1}^{k} \lambda_{i} p_{i} : \ k\in \mathbb{N}, \ p_{i}\in S, \ \lambda_{i}\geq 0\right\}.
	\end{align*}
	It is often more convenient to work with convex cones than convex sets.
	Following the notation of \cite{Sinn2014Algeb-28161}, for any set $S\subseteq \R
	^{n}$, we define $\widehat{S}$ to be the conical hull of $S$ embedded into
	$\R^{n+1}$ at height one:
	\[
		\widehat{S}= \cone(\{ 1 \} \times S ) = \left\{\sum_{i=1}^{k} \lambda_{i} (1,
		p_{i}) : \ k\in \mathbb{N}, \ p_{i}\in S, \ \lambda_{i}\geq 0\right\} \subset
		\R^{n+1}.
	\]

	For any convex cone $C$, the set of linear functions that take only nonnegative
	values on $C$ also forms a convex cone, known as the \emph{dual cone} of $C$,
	denoted $C^{*}$. For $C\subseteq \R^{n}$ this is
	\[
		C^{*} = \{w \in \R^{n} : \langle w,v\rangle \geq 0 \text{ for all }v\in C\}.
	\]
	This is a closed convex cone in $\R^{n}$. It is a classical theorem in convex
	geometry that the dual of $C^{*}$ coincides with the closure of $C$ in the
	Euclidean topology. See Theorem 1.2, the subsequent Problem 2, and Lemma~1.4 of
	\cite[Section IV.1]{BarvinokAlexander2015CiC}. This is known as the biduality
	theorem:

	\begin{theorem}[Biduality]
		For any nonempty convex cone $C$, $(C^{*})^{*} = \overline{C}$.
	\end{theorem}

	Note that when $C = \emptyset$, $C^{*} = \R^{n}$ and
	$(C^{*})^{*} = \{(0,\hdots, 0)\}$. From an inequality description of a cone $C$,
	we can therefore understand its dual.

	\begin{corollary}
		\label{cor:DualConicalHull} Let $K = \{{\bf a}\in \R^{n} : \langle{\bf x}(\theta
		),{\bf a}\rangle \geq 0 \text{ for all }\theta \in \Theta\}$ where ${\bf x}(\theta
		)\in \R^{n}$. The dual cone of $K$ is
		\[
			K^{*} = \overline{\cone\{ {\bf x}(\theta) : \theta\in \Theta\}}.
		\]
	\end{corollary}
	\begin{proof}
		The condition that $\langle{\bf x}(\theta),{\bf a}\rangle \geq 0$ for all $\theta
		\in \Theta$ is equivalent to the condition that $\langle{\bf x},{\bf a}\rangle
		\geq 0$ for all ${\bf x}$ in the conical hull of $\{{\bf x}(\theta) : \theta\in
		\Theta\}$. This shows that $K$ is the dual cone of
		$C = \cone\{{\bf x}(\theta) : \theta\in \Theta\}$. The result then follows from
		the biduality theorem.
	\end{proof}

	\subsection{Duality of plane curves}
	Let $f\in \R[t,x,y]$ be a homogeneous polynomial of degree $n$. That is, $f=\sum
	_{i+j\leq n}c_{ij}t^{i}x^{j}y^{n-i-j}$ for some constants $c_{ij}\in \R$. We
	use $\V(f)$ and $\V_{\R}(f)$ to denote the variety of $f$ in $\PP^{2}(\C)$ and
	$\PP^{2}(\R)$, respectively. Up to scaling, the polynomial $f$ has a unique factorization
	into irreducible polynomials $f = \prod_{i=1}^{s}f_{i}^{m_i}$ where ${\rm gcd}(
	f_{i}, f_{j}) = 1$ for $i\neq j$. We use $f^{\rm red}$ to denote the square-free
	product $\prod_{i=1}^{s}f_{i}$. Then $\V(f) = \V(f^{\rm red})$.

	A point $p\in \V(f)$ is called a nonsingular point of $\V(f)$ if the gradient $\nabla
	(f^{\rm red})$ is nonzero at $p$, in which case
	$\frac{\partial f^{\rm red}}{\partial t}(p)t+\frac{\partial f^{\rm red}}{\partial
	x}(p)x+\frac{\partial f^{\rm red}}{\partial y}(p)y = 0$
	defines the tangent line of $\V(f)$ at $p$. Otherwise we call $p$ a singular point
	of $\V(f)$. We use $\Sing(f)$ and $\Sing_{\R}(f)$ to denote the set of
	singular points of $\V(f)$ and $\V_{\R}(f)$, respectively.

	The \emph{dual variety} of $\V(f)$, denoted $\V(f)^{*}$, is defined as the
	image of $\V(f)\backslash \Sing(f)$ under the map
	$p\mapsto [\nabla f^{\rm red}(p)]$. That is,
	\[
		\V(f)^{*} = \left\{[c:a:b]\in \PP^{2}(\C) : ct+ax+by=0 \text{ is tangent to $\V
		(f)$ at some point $p$}\right\}.
	\]
	When $f$ is irreducible and has degree $\geq 2$, $\V(f)^{*}$ is an irreducible
	plane curve. When $f= ct+ax+by$, the dual variety $\V(f)^{*}$ is a single point
	$[c:a:b]$. In general, if $f = \prod_{i=1}^{s}f_{i}^{m_i}$ is an irreducible factorization
	of $f$, the dual variety of $f$ is the union of the dual varieties of its
	irreducible factors, $\V(f)^{*} = \cup_{i=1}^{s}\V(f_{i})^{*}$.

	\subsection{Coordinate changes}
	Numerical ranges behave nicely under affine-linear maps on $\C\cong \R^{2}$.
	Consider an affine-linear map $L:\R^{2}\to \R^{2}$ given by
	\begin{equation}
		\label{eq:coordChange}L(a,b)= (u_{01}+u_{11}a+u_{21}b, u_{02}+u_{12}a+u_{22}b
		),
	\end{equation}
	where $u_{ij}\in \R$. We can extend this to an action of the affine-linear group
	on $n\times n$ matrices as follows. Given a matrix $A\in \C^{n\times n}$, define
	\[
		L\cdot A = (u_{01}+\ii u_{02})I + (u_{11}+\ii u_{12})\Re(A) + (u_{21}+\ii u_{22}
		)\Im(A).
	\]
	One can check that $L(\Lambda_{k}(A)) = \Lambda_{k}(L\cdot A)$. Note that this
	also corresponds to a linear-change of coordinates on the Kippenhahnn
	polynomial:
	\[
		f_{L\cdot A}(t,x,y) = f_{A}\left(t+u_{01}x + u_{02}y, u_{11}x + u_{12}y, u_{21}
		x+u_{22}y\right).
	\]

	\begin{lemma}
		\label{lm:singDegree} For $A\in \C^{n\times n}$ and $p = (p_{0},p_{1},p_{2})\in
		\R^{3}$ with $(p_{1}, p_{2})\neq (0,0)$, the following quantities coincide:
		\begin{itemize}
			\item[(i)] the corank of the matrix $p_{0}I + p_{1}\Re(A) + p_{2}\Im(A)$,

			\item[(ii)] the maximum $d$ such that $r^{d}$ divides $f_{A}(rq+sp) \in \R[
				r,s]$ for every $q \in \R^{3}$ , and

			\item[(iii)] the minimum degree of a monomial of $f_{L\cdot A}(t,x,1)$ where
				$L:\R^{2}\to \R^{2}$ is any invertible affine linear transformation as in
				\eqref{eq:coordChange} with $(u_{02},u_{12},u_{22})=p$.
		\end{itemize}
	\end{lemma}
	This number is called the \emph{multiplicity} of $f_{A}$ at $p$. The
	multiplicity of $f_{A}$ is $\geq 1$ if and only if $f_{A}(p)=0$ and $\geq 2$
	if and only if both $f_{A}(p)=0$ and $\nabla f_{A}(p)=0$.

	\begin{proof}
		For $(t,x,y)\in \R^{3}$, let $M_{A}(t,x,y)$ denote the linear matrix pencil
		\[
			M_{A}(t,x,y) = tI + x\Re(A) + y\Im(A).
		\]

		(i)$=$(iii) For an invertible affine linear transformation $L:\R^{2}\to \R^{2}$
		as in \eqref{eq:coordChange} with $(u_{02},u_{12},u_{22})=p$, $\Im(L\cdot A)
		= M_{A}(p)$. Therefore
		\[
			f_{L\cdot A}(t,x,y) = \det(tI + x\Re(L\cdot A) + y M_{A}(p)).
		\]
		By the Laplace expansion of the determinant, we see that the degree of $f_{L\cdot
		A}$ in the variable $y$, denoted $\deg_{y}(f_{L\cdot A})$, is $\leq{\rm rank}
		(M_{A}(p))$. Conversely, the restriction $f_{L\cdot A}(t,0,y)$ factors as $t^{d}
		\prod_{i}(t+\lambda_{i} y)$ where the product is taken over the nonzero eigenvalues
		of $M_{A}(p)$. This shows that
		$\deg_{y}(f_{L\cdot A}) \geq{\rm rank}(M_{A}(p))$. Since $f_{L\cdot A}(t,x,y)$
		is homogeneous of degree $n$, $n - \deg_{y}(f_{L\cdot A})$ is the quantity
		in (iii). Similarly $n-{\rm rank}(M_{A}(p))$ is the quantity in (i).

		(i)$=$(ii) Similarly, for $q\in \R^{3}$, $f_{A}(rq+sp) = \det(rM_{A}(q) + s M
		_{A}(p))$. By the Laplace expansion of the determinant, we see that the degree
		of $s$ in this polynomial is $\leq{\rm rank}(M_{A}(p))$. Since it is
		homogeneous of degree $n$ in $r,s$, it is therefore divisible by $r^{d}$ where
		$d = n -{\rm rank}(M_{A}(p))$. Conversely, for $q = (1,0,0)$, this
		restriction factors as $r^{d}\prod_{i}(r+\lambda_{i} s)$ where the product
		is taken over the nonzero eigenvalues of $M_{A}(p)$.
	\end{proof}

	\section{The dual convex cone of rank-$k$ numerical range}
	\label{sec:DualCone}

	Li and Sze \cite{Li2008} give a description \eqref{eq:half_plane} of $\Lambda_{k}
	(A)$ as the intersection of halfplanes. Taking $\mu = a+\ii b$, we rewrite this
	description as
	\begin{equation}
		\label{eq:half_plane2}\Lambda_{k}(A) = \left\{a+\ii b\in \C : a \cos(\theta)
		+ b \sin(\theta) \leq \lambda_{k}(\theta) \text{ for all }\theta\in [0,2\pi)
		\right\}
	\end{equation}
	where $\lambda_{k}(\theta)$ is short-hand notation for $\lambda_{k}(\cos(\theta
	) \Re A + \sin(\theta) \Im A).$ Using this description, we use convex duality
	to study $\Lambda_{k}(A)$. For convenience, we lift $\Lambda_{k}(A)$ to a
	convex cone in $\R^{3}$. Namely, let
	\[
		\widehat{\Lambda_k(A)}=\{(\lambda, \lambda a, \lambda b):a+\ii b\in \Lambda_{k}
		(A), \lambda \in \R_{\geq 0}\}
	\]
	when $\Lambda_{k}(A)$ is nonempty. When $\Lambda_{k}(A)$ is empty, we use the
	convention $\widehat{\Lambda_k(A)}= \{(0,0,0)\}$.

	To understand the dual cone, consider the closed curve in $\R^{3}$ parameterized
	by
	\begin{equation}
		\label{eq:Ok(A)}O_{k}(A) = \{(\lambda_{k}(\theta), -\cos \theta, -\sin\theta
		):\theta \in [0,2\pi) \} .
	\end{equation}
	This curve is contained in the set of points $(t,x,y)\in \R^{3}$ with $f_{A}(t,
	x,y)=0$ and $x^{2}+y^{2}=1$.

	\begin{theorem}
		\label{th:dual_cone} The convex cone $\widehat{\Lambda_k(A)}$ has the inequality
		description
		\[
			\widehat{\Lambda_k(A)}= \{(c,a,b) : c\geq 0 \text{ and }\cos(\theta)a+\sin(
			\theta)b\leq \lambda_{k}(\theta)c\text{ for all }\theta\in [0,2\pi)\}
		\]
		and its dual cone is given by
		\[
			(\widehat{\Lambda_k(A)})^{*} = \overline{\cone(\{(1,0,0)\} \cup O_k(A))}.
		\]
	\end{theorem}

	\begin{proof}
		From \Cref{cor:DualConicalHull}, it suffices to prove the inequality
		description for $\widehat{\Lambda_k(A)}$ and the description of
		$(\widehat{\Lambda_k(A)})^{*}$ follows. The containment $\subseteq$ follows directly
		from the description of \eqref{eq:half_plane} of Li and Sze.

		For the reverse containment, suppose that $(c,a,b)$ satisfies $c\geq 0$ and $\cos
		(\theta)a+\sin(\theta)b\leq \lambda_{k}(\theta)c\text{ for all }\theta\in [0,
		2\pi)$. If $c >0$, then $\frac{1}{c}(a,b)$ satisfies all the inequalities that
		define $\Lambda_{k}(A)$ and so $(c,a,b) = c(1,a/c,b/c)$ belongs to $\widehat{\Lambda_k(A)}$.
		If $c = 0$, we claim that $(a,b) = (0,0)$. Taking
		$\theta \in \{0, \pi/2, \pi, 3\pi/2\}$, we see that $(a,b)$ must satisfy the
		inequalities $-a \geq 0$, $-b \geq 0$, $a \geq 0$ and $b\geq 0$, which imply
		$(a,b) = (0,0)$. Therefore $(c,a,b) = (0,0,0)\in \widehat{\Lambda_k(A)}$.
	\end{proof}

	\begin{remark}
		Since $\Lambda_{k}(A)$ is compact, $\widehat{\Lambda_k(A)}$ is a closed, pointed
		cone. That is, it is closed and does not contain any line through the origin.
	\end{remark}

	\begin{corollary}
		A point $a+\ii b$ belongs to $\Lambda_{k}(A)$ if and only if the linear function
		$t+ax + by$ is nonnegative on $O_{k}(A)$.
	\end{corollary}
	\begin{proof}
		This follows directly from the description \eqref{eq:half_plane} of Li and
		Sze.
	\end{proof}

	The following example shows why the inclusion of $(1,0,0)$ in \Cref{th:dual_cone}
	is necessary:
	\begin{example}
		\label{ex:OkPlane} Consider the $3\times 3$ matrix $A =
		\begin{pmatrix}
			0 & 0 & 0 \\
			0 & 0 & 2 \\
			0 & 0 & 0
		\end{pmatrix}$. The Kippenhahn polynomial is
		$f_{A}(t,x,y) = t(t^{2} -x^{2} - y^{2})$. One can check that $\lambda_{2}(\theta
		) = 0$ for all $\theta \in [0,2\pi)$ and therefore $\Lambda_{2}(A) = \{0\}$.
		The conical hull of $O_{2}(A)$ is therefore the plane $\{c = 0\}$ and so the
		dual cone of this conical hull is the whole real line spanned by $(1,0,0)$. This
		strictly contains $\widehat{\Lambda_2(A)}$, which is the ray
		$\{(c,0,0):c\geq 0\}$. The conical hull of $\{(1,0,0)\}\cup O_{2}(A)$ is the
		halfspace $\{(t,x,y): t\geq 0\}$, whose dual cone is
		$\widehat{\Lambda_2(A)}$, as promised by \Cref{th:dual_cone}.
	\end{example}

	In \Cref{ex:OkPlane}, the curve $O_{k}(A)$ was contained in a plane. As we will
	see below, this is the only way for $(\widehat{\Lambda_k(A)})^{*}$ to differ
	from the closure of the conical hull of $O_{k}(A)$. To see this, we prove some
	very simple but useful inequalities.

	\begin{lemma}
		\label{le:antipodal_eigenvalue} For any $\theta\in [0,2\pi)$,
		\begin{align*}
			\lambda_{k}(\theta) + \lambda_{k}(\theta+\pi) & \geq 0 \text{ for }k \leq (n+1)/2, \text{ and } \\
			\lambda_{k}(\theta) + \lambda_{k}(\theta+\pi) & \leq 0 \text{ for }k \geq (n+1)/2.
		\end{align*}
	\end{lemma}
	\begin{proof}
		Note that $k \leq (n+1)/2$ if and only if $k \leq n-k+1$. By definition, we have
		that $\lambda_{k}(\theta) \geq \lambda_{n-k+1}(\theta)$. Also we have
		$\lambda_{k}(\theta + \pi) = -\lambda_{n-k+1}(\theta).$ Together this gives
		\[
			\lambda_{k}(\theta) + \lambda_{k}(\theta+\pi) = \lambda_{k}(\theta) - \lambda
			_{n-k+1}(\theta) \geq 0.
		\]
		The second statement follows analogously by taking the inequality
		$\lambda_{k}(\theta) \leq \lambda_{n-k+1}(\theta)$.
	\end{proof}

	\begin{corollary}
		\label{cor:Duals} Suppose that $O_{k}(A)$ is not contained in a plane through
		the origin in $\R^{3}$. If $k \leq (n+1)/2$, then
		\[
			(\widehat{\Lambda_k(A)})^{*} = \overline{\cone(O_k(A))}
		\]
		and if $k \geq (n+1)/2$, then $(\widehat{\Lambda_k(A)})^{*} = \R^{3}$, $\widehat
		{\Lambda_k(A)}= \{(0,0,0)\}$, and $\Lambda_{k}(A)=\emptyset$.
	\end{corollary}

	\begin{proof}
		We show that when $k \leq (n+1)/2$, $(1,0,0)$ belongs to the closed convex
		cone $C = \overline{\cone(O_k(A))}$. Suppose, for the sake of contradiction,
		that it does not. Then there exists a hyperplane separating these sets. That
		is, there exists $(c,a,b)\in \R^{3}$ with $ct+ax+by \geq 0$ for all
		$(t,x,y) \in C$ and $c<0$.

		Since $O_{k}(A)$ is not contained in the plane $\{(t,x,y): ct+ax+by=0\}$, there
		must be some point of $O_{k}(A)$ not on this plane. That is, there is some $\theta
		\in [0,2\pi)$ for which
		\[
			\lambda_{k}(\theta)c -\cos(\theta)a-\sin(\theta)b > 0.
		\]
		Since $ct+ax+by$ is nonnegative on $O_{k}(A)$, we also have that
		\[
			\lambda_{k}(\theta+\pi)c + \cos(\theta)a+\sin(\theta)b = \lambda_{k}(\theta
			+\pi)c - \cos(\theta +\pi)a-\sin(\theta+\pi)b \geq 0.
		\]
		Adding these two inequalities gives that
		$c(\lambda_{k}(\theta) + \lambda_{k}(\theta + \pi)) > 0$. Since $c<0$, this
		contradicts the first inequality in \Cref{le:antipodal_eigenvalue} when
		$k \leq (n+1)/2$. Therefore $(1,0,0)\in \overline{\cone(O_k(A))}$. The claim
		then follows from \Cref{th:dual_cone}.

		For $k \geq (n+1)/2$, we use that $O_{k}(A) = - O_{n-k+1}(A)$ and $n-k+1\leq
		(n+1)/2$. By the arguments above, the point $(-1,0,0)$ belongs to $C$. The
		cone $\overline{\cone(\{(1,0,0)\}\cup O_k(A))}$ therefore contains the whole
		real line $\{(t,0,0): t \in \mathbb{R}\}$. By \Cref{th:dual_cone}, this cone
		coincides with $(\widehat{\Lambda_k(A)})^{*}$. From this and the parametrizaton
		of $O_{k}(A)$, we see that it also contains the set
		$\{(0,-\cos(\theta), -\sin(\theta)) : \theta\in [0,2\pi)\}$. It follows that
		$(\widehat{\Lambda_k(A)})^{*}$ is all of $\R^{3}$. By the biduality theorem,
		$\widehat{\Lambda_k(A)}= \{(0,0,0)\}$. It follows that
		$\Lambda_{k}(A)=\emptyset$.
	\end{proof}

	\begin{example}
		\label{ex:pringle} The $4\times 4$ symmetric matrix
		\[
			A =
			\begin{pmatrix}
				0 & 2   & 0   & 0   \\
				2 & 0   & \ii & 0   \\
				0 & \ii & 0   & 1   \\
				0 & 0   & 1   & 0\

			\end{pmatrix}
			\ \text{ gives }\  f_{A} =\det
			\begin{pmatrix}
				t       & 2x      & 0       & 0       \\
				\! 2x\! & t       & \! y \! & 0       \\
				0       & \! y \! & t       & \! x \! \\
				\!0 \!  & 0       & \! x \! & t
			\end{pmatrix}
			= t^{4} - 5 t^{2} x^{2} + 4 x^{4} - t^{2} y^{2}.
		\]
		\Cref{fig:pringle} shows the curves $O_{k}(A)$ for $k=1,2,3,4$ as well as the
		convex hull of $O_{2}(A)$. Every point $(t,x,y)$ in the curve $O_{1}(A)$
		satisfies $t\geq 1$. Its convex hull does not contain the origin and its conical
		hull is a pointed convex cone given by
		\[
			(\widehat{\Lambda_1(A)})^{*} = \{(t,x,y)\in \R^{3} : f_{A}(t,x,y)\geq 0, 2
			t^{2} - 5 x^{2} - y^{2}\geq 0, t\geq0 \}.
		\]
		On the other hand, the convex hull of $O_{2}(A)$ has a family of edges whose
		tangent planes approach $\{c=0\}$. The plane $\{c=0\}$ meets this convex body
		in an edge between the two points $(0,\pm1,0)$. The \emph{conical hull} of $O
		_{2}(A)$ is therefore $\{(t,x,y): t>0\} \cup \{(0,x,0): x\in \R\}$. The
		closure is the half-plane
		$(\widehat{\Lambda_2(A)})^{*} = \{(t,x,y): t\geq0\}$. The dual cone is the ray
		$\widehat{\Lambda_2(A)}= \{(c,0,0): c\geq 0\}$, giving
		$\Lambda_{2}(A) = \{0\}\subset \C$.
	\end{example}

	\begin{figure}
		\includegraphics[height=2in]{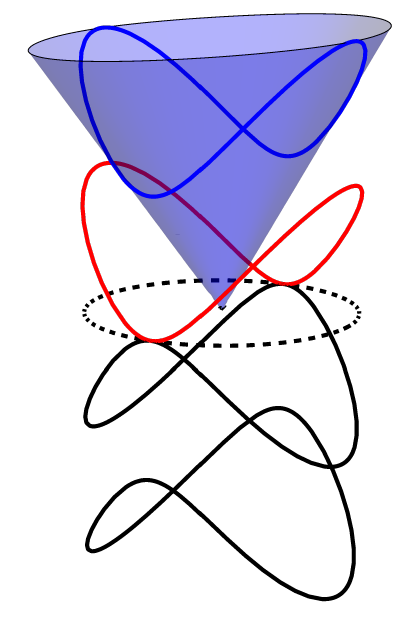}
		\includegraphics[height=2in]{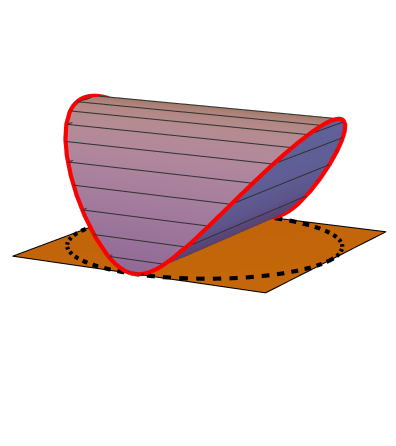}
		\includegraphics[height=2in]{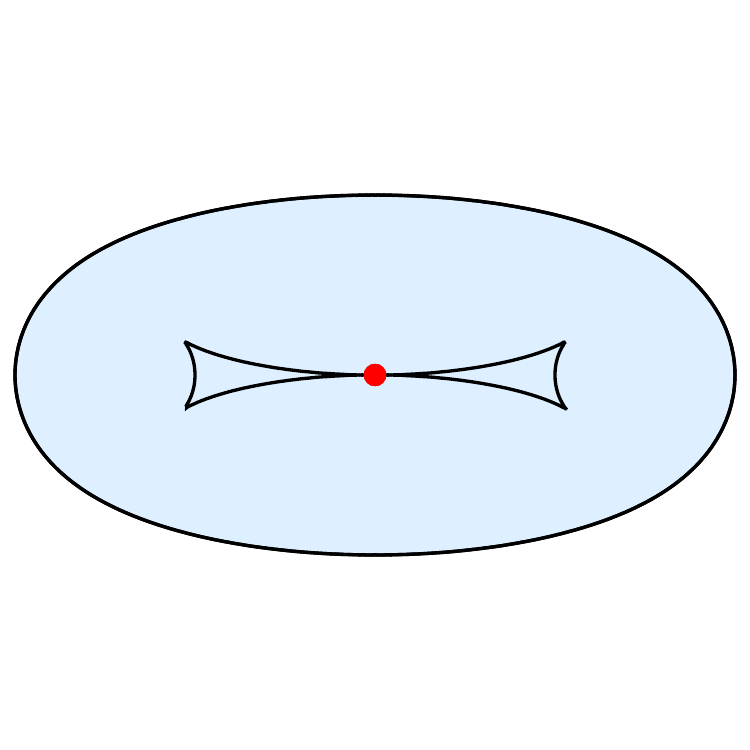}
		\caption{The curves $O_{k}(A)$ from \Cref{ex:pringle} with the cone
		$(\widehat{\Lambda_1(A)})^{*}$, the convex hull of $O_{2}(A)$, and higher-rank
		numerical ranges $\Lambda_{1}(A)$ and $\Lambda_{2}(A)$.}
		\label{fig:pringle}
	\end{figure}

	\section{A membership test for the rank-$k$ numerical range}
	\label{sec:MembershipTest}

	In this section we describe an algorithm for testing membership of a given
	point $a+\ii b$ in $\Lambda_{k}(A)$. This relies on the description of the dual
	cone from \Cref{th:dual_cone}. Note that a point $(1,a,b)$ belongs to
	$\widehat{\Lambda_k(A)}$ if and only if the functional $\ell(t,x,y) = t+ax+by$
	is nonnegative on $O_{k}(A)$. This results in many different equivalent conditions
	for membership in $\Lambda_{k}(A)$.

	\begin{theorem}[Membership test]
		\label{thm:MembershipTest} Let $(a,b) \in \mathbb{R}^{2}$ and $\ell= t +ax +
		by$. The following are equivalent:
		\begin{itemize}
			\item[(a)] $a+ \ii b \in \Lambda_{k}(A)$,

			\item[(b)] $tI_{n} +x\Re A + y \Im A$ has at most $n-k$ strictly positive eigenvalues
				for all $(t,x,y) \in \mathcal{V}_{\mathbb{R}}(\ell)$,

			\item[(c)] $tI_{n} +x \Re A + y \Im A$ has at most $n-k$ strictly negative
				eigenvalues for all $(t,x,y) \in \mathcal{V}_{\mathbb{R}}(\ell)$,

			\item[(d)] $\lambda_{n-k+1}(x\Re A + y\Im A ) \leq ax + by$ for all $(x,y)
				\in \mathbb{R}^{2}$,

			\item[(e)] $\lambda_{k}( x \Re A + y\Im A ) \geq ax + by$ for all $(x,y) \in
				\mathbb{R}^{2}$, and

			\item[(f)] $\lambda_{k}( \Re A + y\Im A ) \geq a + by$ and $\lambda_{n-k+1}
				( \Re A + y\Im A ) \leq a + by$ for all $y \in \mathbb{R}$.
		\end{itemize}
	\end{theorem}

	\begin{proof}
		(a)$\Leftrightarrow$(e) The $\Rightarrow$ direction follows directly from the
		description \eqref{eq:half_plane2} of $\Lambda_{k}(A)$. For the converse,
		suppose $(x,y)\in \R^{2}$. We can write $(x,y) = r(\cos\theta,\sin\theta)$ for
		some $\theta \in [0,2\pi)$ and $r\in \R_{\geq 0}$. By \eqref{eq:half_plane2},
		$\lambda_{k}(\theta)\geq a\cos(\theta) + b\sin(\theta)$. Rescaling both sides
		by $r$ gives the claim.

		(d)$\Leftrightarrow$(e) For any Hermitian matrix $M$,
		$\lambda_{k}(-M) = -\lambda_{n-k+1}(M)$. Since $\R^{2}$ is closed under negation
		and $(-x)\Re A + (-y)\Im A = -(x\Re A + y\Im A )$, the claim follows.

		(d,e)$\Leftrightarrow$(f) Specializing to $x=1$ gives the forward direction.
		Conversely, suppose that (f) holds and take $(x,y)\in \R^{2}$. If $x > 0$,
		then
		\[
			\lambda_{k}(x \Re A + y \Im A) = x\lambda_{k}(\Re A + (y/x) \Im A) \geq x(a
			+b(y/x)) = ax +by.
		\]
		The inequality follows from (f). Similarly, if $x < 0$, then
		\begin{align*}
			\lambda_{k}(x \Re A + y \Im A) & = -x\lambda_{k}(-\Re A - (y/x) \Im A)                               \\
			                               & = x \lambda_{n-k+1}(\Re A + (y/x) \Im A) \geq x(a+b(y/x)) = ax +by.
		\end{align*}
		Finally, we note that $\lambda_{k}(x \Re A + y \Im A)$ is a continuous
		function of $(x,y)$. Since it is $\geq ax+by$ for all $(x,y)$ with $x\neq 0$,
		it is $\geq ax+by$ for any $(x,y)$. This shows that condition (e) holds,
		which implies (d) as above.

		(c)$\Leftrightarrow$(e) Condition (c) is equivalent to the condition that
		$\lambda_{k}(tI_{n} +x \Re A + y \Im A) \geq 0$ for all
		$(t,x,y)\in \mathcal{V}_{\mathbb{R}}(\ell)$. Note that $(t,x,y)\in\mathcal{V}
		_{\mathbb{R}}(\ell)$ if and only if $t = -ax-by$. We see that
		\[
			\lambda_{k}((-ax-by)I_{n} +x \Re A + y \Im A) = -ax-by + \lambda_{k}(x \Re
			A + y \Im A).
		\]
		This is $\geq 0$ if and only if $\lambda_{k}(x \Re A + y \Im A) \geq ax+by$.

		(b)$\Leftrightarrow$(c) Since $\mathcal{V}_{\mathbb{R}}(\ell)$ is closed
		under scaling, this follows from $\lambda_{k}(-M) = -\lambda_{n-k+1}(M)$.
	\end{proof}

	These equivalences reduce the problem of deciding whether an element $a + \ii b$
	belongs to $\Lambda_{k}(A)$ to checking the signs of the eigenvalues in the matrix
	pencil $- (a+by) I_{n}+ \Re A + y\Im A$. An eigenvalue can only change signs
	at the roots of the determinant
	$\det(-(a+by)I_{n}+\Re A + y\Im A ) = f_{A}(-a-by,1,y)$. It therefore suffices
	to check the signature of this matrix pencil in between the roots of this univariate
	polynomial.

	One subtlety is that the restriction of $f_{A}$ to the line $t+ax+by=0$ could be
	identically zero, in which case $\ell = t+ax+by$ is a factor of $f_{A}$. To deal
	with this, we first factor out all powers of $\ell$ from $f_{A}$. This gives the
	following algorithm:

	\begin{algorithm}
		[H]
		\caption{Membership test
		\smallskip
		\\
		\textbf{Input:} $A\in \C^{n\times n}$, $k \in \mathbb{Z}$ with
		$1\leq k\leq n$, $a+\ii b \in \mathbb{C}$ \\
		\textbf{Output:} ``True'' if $a+\ii b \in \Lambda_{k}(A)$, ``False'' otherwise.
		\smallskip
		}
		\label{alg:membership}
		\begin{algorithmic}
			[0] \State Take $f= \det(tI_{n} + x\Re A + y\Im A)$ and $\ell= t + ax + by$.
			\State Let $d= \max \{e \in \mathbb{N}: \ell^{e} \text{ divides }f\}$ and
			define $\tilde{f}:= f/\ell^{d}$. \State Let $h(y) = \tilde{f}(-a-by,1,y) \in
			\mathbb{R}[y]$ and compute the distinct real roots $r_{1} < \ldots < r_{m}$
			of $h$. \If{ $h(y)$ has no real roots} \State define $s_{0}=0$ \Else \State
			define $s_{0} = r_{1}-1$ and $s_{m} = r_{m} + 1$. \For{$i = 1, \hdots, m-1$}
			\State define $s_{i} = (r_{i} + r_{i+1})/2$. \EndFor \EndIf \For{$i = 0, \hdots, m$}
			\If{ $\lambda_{k}(\Re A + s_{i} \Im A ) < a+ bs_{i} \text{ or }\lambda_{n-k+1}(\Re A + s_{i} \Im A ) > a + bs_{i}$ }
			\State\Return{``False''} and {\bf stop} \EndIf \EndFor \State \Return{``True''}
		\end{algorithmic}
	\end{algorithm}

	\begin{proposition}
		\Cref{alg:membership} correctly determines membership in $\Lambda_{k}(A)$.
	\end{proposition}

	\begin{proof}
		If the algorithm returns ``False'', then, for some $i$,
		$\lambda_{k}(\Re A + s_{i} \Im A ) < a+ bs_{i}$ or
		$\lambda_{n-k+1}(\Re A + s_{i} \Im A ) > a + bs_{i}$. By \Cref{thm:MembershipTest},
		$a+\ii b$ does not belong to $\Lambda_{k}(A)$.

		Conversely, if $a+\ii b\not\in \Lambda_{k}(A)$, then
		$\lambda_{k}(\Re A + s \Im A ) < a+ bs \text{ or }\lambda_{n-k+1}(\Re A + s \Im
		A ) > a + bs$
		for some $s\in \R$. Since this is an open condition, we can assume that $s$
		is not a root of $h$. It follows that $s$ belongs to some connected component
		$I$ of $\R\backslash \{r_{1}, \hdots, r_{m}\}$, which has some
		representative $s_{i}\in I$. We claim that, for any $j$, the value of
		\[
			\lambda_{j}((-a-by)I_{n} + \Re A + y\Im A) = -a-by + \lambda_{j}(\Re A + y\Im
			A)
		\]
		is either identically zero or has constant sign on the open interval $I$.
		Note that by definition
		\[
			f(t-a-by, 1,y) = t^{d}\tilde{f}(t-a-by, 1,y) = \det((t-a-by)I_{n} + \Re A +
			y\Im A).
		\]
		For given $y$, the roots of this polynomial (in $t$) are the negatives of the
		eigenvalues of the matrix $(-a-by)I_{n} + \Re A + y\Im A$. The values of $y$
		for which $t=0$ is a root of multiplicity $>d$ are exactly the roots of
		$h(y) = \tilde{f}(-a-by, 1,y)$. Therefore for $y\in I$, the roots of
		$\tilde{f}(-a-by, 1,y)$ are nonzero and have constant sign. It follows that for
		$y\in I$, $\lambda_{j}((-a-by)I_{n} + \Re A + y\Im A)$ is either identically
		zero or has constant sign. In particular, since $I$ contains $s$ and $s_{i}$,
		$\lambda_{k}(\Re A + s_{i} \Im A ) < a+ bs_{i} \text{ or }\lambda_{n-k+1}(\Re
		A + s_{i} \Im A ) > a + bs_{i}$ and the algorithm returns ``False''.
	\end{proof}

	We finish this section by illustrating the membership test on a small example.

	\begin{example}
		\label{ex:circleandline} Consider the $3\times 3$ matrix $A =
		\begin{pmatrix}
			1 & 0 & 0   \\
			0 & 0 & 1/2 \\
			0 & 0 & 0
		\end{pmatrix}$, $k=2$, and $a+\ii b = 1+0 \ii$. The Kippenhahn polynomial is
		$f(t,x,y) = \frac{1}{16}(t+x)(16t^{2} - x^{2}-y^{2})$. We take $\ell = t+ax+b
		y= t+x$. Note that $\ell$ is the highest power of $\ell$ that divides $f$, so
		$\tilde{f}= f/\ell = \frac{1}{16}(16t^{2} - x^{2}-y^{2})$. Then
		\[
			h(y) = \tilde{f}(-1,1,y) = \frac{1}{16}(16(-1)^{2} - 1^{2}-y^{2}) = \frac{1}{16}
			(15-y^{2})
		\]
		which has $m=2$ real roots, $r_{1}=-\sqrt{15}$ and $r_{2}=\sqrt{15}$. From
		this, we define three test points $s_{0} = -\sqrt{15}-1$, $s_{1} = 0$, and $s
		_{2} = \sqrt{15}+1$. For any $y\in \R$, the matrix
		\[
			(-a-by)I_{n} + \Re A + y\Im A =
			\begin{pmatrix}
				0 & 0                 & 0                 \\
				0 & -1                & \frac{1-\ii y}{4} \\
				0 & \frac{1+\ii y}{4} & -1
			\end{pmatrix}
		\]
		has eigenvalues $0,-1\pm\frac{1}{4}\sqrt{1+y^{2}}$. In particular, for $y=s_{1}
		= 0$, it has two negative eigenvalues. By \Cref{thm:MembershipTest}, we
		conclude that $1\not\in\Lambda_{2}(A)$. The curve $\V(f_{A})$ is shown in \Cref{fig:lineandcircle}
		in the chart $\{t=1\}$. Connected components of the complement are labelled
		with the signs of the eigenvalues of the matrix $I_{n} + x\Re A + y\Im A$.
	\end{example}

	\section{The algebraic boundary of $\Lambda_{k}(A)$}
	\label{sec:algBoundary}

	In this section we define a polynomial $g_{A}\in \R[a,b]$ that vanishes on the
	boundary of $\Lambda_{k}(A)$ for any $k$. This polynomial will be well-defined
	whenever $f_{A}$ is not a power of a linear form. We define $g_{A}$ as the
	minimal polynomial vanishing on the dual curve of $\V(f_{A})$ and all lines
	coming from singularities of $f_{A}$. Formally, we compute $g_{A}$ as follows.
	Let $f = f_{A}^{\rm red}$ denote the squarefree part of the factorization of $f
	_{A}$, i.e.~the minimal polynomial vanishing on $\V(f_{A})$. Consider the
	ideal $I\subset \R[t,x,y,a,b]$ given by
	\[
		I = \langle f, \ t +ax + by \rangle + \left\langle 2\times 2 \ {\rm minors}
		\begin{pmatrix}
			1                             & a                             & b                             \\
			\frac{\partial f}{\partial t} & \frac{\partial f}{\partial x} & \frac{\partial f}{\partial y}
		\end{pmatrix}\right\rangle
	\]
	and the ideal $J$ obtained from $I$ by saturating by the irrelevant ideal
	$\langle t,x,y\rangle$ and then eliminating the variables $(t,x,y)$:
	\[
		J ={\rm radical}\left({\rm eliminate}\left(I:\langle t,x,y\rangle^{\infty}, \{
		t,x,y\}\right)\right).
	\]

	We claim that this is a principal ideal whose variety consists of the dual
	curve of $\V(f_{A})$ and lines corresponding to singularities of $\V(f_{A})$.

	\begin{lemma}
		\label{lem:g_A well defined} If $f_{A}$ is not a power of a linear form,
		then the ideal $J$ above is principal and
		\begin{equation}
			\label{eq:VgA}\V(J) = \V(f_{A})^{*} \cup \bigcup_{[p_0:p_1:p_2]\in \Sing(f_A)}
			\V(p_{0}+p_{1}a+p_{2}b).
		\end{equation}
	\end{lemma}

	\begin{proof}
		We first show that the variety of $J$ is as claimed. As in the construction
		of $J$, let $f = f_{A}^{\rm red}$ denote the squarefree part of the factorization
		of $f_{A}$. By construction, the variety of $J$ (over $\C$) consists of the set
		of $(a,b)\in \C^{2}$ for which there exists $p = (p_{0},p_{1},p_{2})\in \C^{3}
		\backslash \{(0,0,0)\}$ and $\lambda\in \C$ with $f(p)=0$,
		$p_{0}+ap_{1}+b p_{2}=0$, and $\lambda(1,a,b) = \nabla f(p)$. Under the
		additional restriction that $\nabla f(p)$ is nonzero, this is exactly the dual
		curve of $\V(f)$. We see that for any singular point $p$, any point $(a,b)\in
		\C^{2}$ satisfying $p_{0}+p_{1}a+p_{2}b=0$ belongs to $\V(J)$, since
		$0\cdot(1,a,b) = \nabla f(p)$.

		To show that $\V(J)$ is principal, it suffices to show that its variety is a
		union of plane curves. Suppose
		$f = \ell_{1}\cdots \ell_{r} \cdot f_{1}\cdots f_{s}$ where
		$\ell_{1}, \hdots, \ell_{r}$ are linear forms and $f_{1}, \hdots, f_{s}$
		have degree $\geq 2$. The dual curve of $f$ is the union of the dual curves of
		its factors, i.e.~$\V(f)^{*} = \left(\cup_{i=1}^{r} \V(\ell_{i})^{*}\right) \cup
		\left(\cup_{j=1}^{s} \V(f_{j})^{*}\right)$. For each $i$ and $j$, $\V(\ell_{i}
		)^{*}$ is a point and $\V(f_{j})^{*}$ is a curve. It suffices to show that
		each point $\V(\ell_{i})^{*}$ is contained in $\V(p_{0}+p_{1}a+p_{2}b)$ for
		some singular point $p$ of $\V(f)$. To do this, consider the factorization
		$f = \ell_{i} \cdot h$. By assumption, $f_{A}$ is not a power of a linear
		form and so its reduced polynomial $f$ is not linear, implying that $\deg(h)\geq
		1$. It follows that the intersection of $\V(\ell_{i})$ and $\V(h)$ in
		$\PP^{2}(\C)$ is nonempty. Let $p$ be a point in this intersection. Then $p$
		will be a singular point of $\V(f)$. Moreover, since $\ell_{i}(p) =0$, the point
		$\V(\ell_{i})^{*}$ is contained in the line $\V(p_{0}+p_{1}a+p_{2}b)$.
		Therefore the variety of $J$ is
		\[
			\V(J) = \left(\bigcup_{j=1}^{s} \V(f_{j})^{*}\right) \cup \left(\bigcup_{[p_0:p_1:p_2]\in
			\Sing(f_A)}\V(p_{0}+p_{1}a+p_{2}b)\right),
		\]
		which is a union of plane curves. It follows that $J$ is a principal ideal.
	\end{proof}

	\begin{definition}[The polynomial $g_{A}$]
		\label{def:g_A} Let $A\in \C^{n\times n}$ and suppose that $f_{A}$ is not a power
		of a linear form. Define the polynomial $g_{A} \in \R[a,b]$ to be the
		generator of the ideal $J$ above. When $\V(f_{A})$ is singular, the variety
		of $g_{A}$ includes both the dual curve of $\V(f_{A})$ and lines
		corresponding to its singular points.
	\end{definition}

	\begin{lemma}
		\label{lem:lambda_diff} Let $f = f_{A}^{\rm red}$. If $\nabla f$ is nonzero
		at $p =(\lambda_{k}(x^{*},y^{*}), -x^{*},-y^{*})$ for some
		$(x^{*},y^{*})\in \R^{2}$, then the function $\lambda_{k}(x,y)$ is differentiable
		at $(x^{*},y^{*})$ with
		\[
			\frac{\partial \lambda_{k}}{\partial x}(x^{*},y^{*}) = \frac{\frac{\partial
			f}{\partial x}(p)}{\frac{\partial f}{\partial t}(p)}\ \ \text{ and }\ \ \frac{\partial
			\lambda_{k}}{\partial y}(x^{*},y^{*}) = \frac{\frac{\partial f}{\partial y}(p)}{\frac{\partial
			f}{\partial t}(p)}.
		\]
	\end{lemma}

	\begin{proof}
		Suppose that $\nabla f(p)$ is nonzero. By \cite[Lemma~2.4]{PLAUMANN201348}, $\frac{\partial
		f}{\partial t}$ is nonzero at $p$. Note that for any $(x,y)\in \R^{2}$, the point
		$(\lambda_{k}(x,y), -x,-y)$ belongs to the variety of $f$. The result then
		follows from the implicit function theorem.
	\end{proof}

	\begin{theorem}
		\label{thm:boundary} If $a+ \ii b$ belongs to $\partial\Lambda_{k}(A)$, then
		$g_{A}(a,b)=0$.
	\end{theorem}

	\begin{proof}
		Since $a+\ii b\in \Lambda_{k}(A)$, the linear function
		$\ell = t + a x + b y$ is nonnegative on the curve
		$O_{k}(A) = \{(\lambda_{k}(\theta), -\cos(\theta), -\sin(\theta)) : \theta\in
		[0, 2\pi)\} \subset \V(f_{A})$. Moreover, since it belongs to the boundary
		of $\Lambda_{k}(A)$, $\ell$ must be zero at some point $p = (p_{0},p_{1},p_{2}
		)\in O_{k}(A)$ (otherwise $\ell$ is positive on $O_{k}(A)$ and any small
		enough perturbation of $\ell$ would remain positive).

		If $p$ is a singular point of $\V(f_{A})$, then the linear form
		$p_{0}+p_{1}a+p_{2}b$ is a factor of the polynomial $g_{A}$. Since $0=\ell(p)
		= p_{0} + a p_{1} + b p_{2}$, $(a,b)$ belongs to the variety of $g_{A}$.

		Now, suppose that $\nabla f_{A}^{\rm red}(p)$ is nonzero and consider $F_{(a,b)}
		:\R^{2}\to \R$ given by function
		\[
			F_{(a,b)}(x,y) = \ell(\lambda_{k}(x,y), -x, -y) = \lambda_{k}(x,y) -ax -by.
		\]
		Since $\lambda_{k}(\mu x,\mu y) = \mu\lambda_{k}( x,y)$ for any $\mu \in \R_{>0}$
		and $\ell$ is nonnegative on $O_{k}(A)$, $F_{(a,b)}$ is nonnegative on all
		of $\R^{2}$. Moreover, it is zero at $(p_{1},p_{2})$. By Lemma~\ref{lem:lambda_diff},
		$F_{(a,b)}$ is also differentiable at $(p_{1}, p_{2})$ with
		\[
			\frac{\partial F_{(a,b)}}{\partial x}(p_{1}, p_{2}) = \frac{\frac{\partial
			f_{A}^{\rm red}}{\partial x}(p)}{\frac{\partial f_{A}^{\rm red}}{\partial
			t}(p)}-a \ \ \ \ \text{ and }\ \ \ \
\frac{\partial F_{(a,b)}}{\partial y}(p
			_{1}, p_{2}) = \frac{\frac{\partial f_{A}^{\rm red}}{\partial y}(p)}{\frac{\partial
			f_{A}^{\rm red}}{\partial t}(p)}-b.
		\]
		Since $F_{(a,b)}$ achieves its minimum value at $(p_{1},p_{2})$, these derivatives
		must be zero, from which we conclude that $(1,a,b)$ is a scalar multiple of
		$\nabla f_{A}^{\rm red}(p)$. Since $p \neq 0$ belongs to $\V(f_{A})$, it
		follows by definition that $(1,a,b)$ belongs to $\mathcal{V}(f_{A})^{*}$.
	\end{proof}

	\begin{example}[\Cref{ex:circleandline} continued]
		Consider the $3\times 3$ matrix $A$ from \Cref{ex:circleandline} with $f_{A}(
		t,x,y) = \frac{1}{16}(t+x)(16t^{2} - x^{2}-y^{2})$. The dual curve of
		$\V(f_{A})$ is the union of the duals of its two components
		$\V(f_{A})^{*} = \{[1:1:0]\} \cup \V(1 - 16a^{2}-16b^{2})$. The curve $\V(f_{A}
		)$ has two singular points, at which the two factors are both zero, namely $[
		1:-1:\pm\sqrt{15}]$. These both contribute a linear factor to $g_{A}$, giving
		\[
			g_{A}(a,b) = (1 - 16a^{2}-16b^{2})(1-a+\sqrt{15}b)(1-a-\sqrt{15}b).
		\]
	\end{example}

	\begin{figure}
		\includegraphics[height=2in]{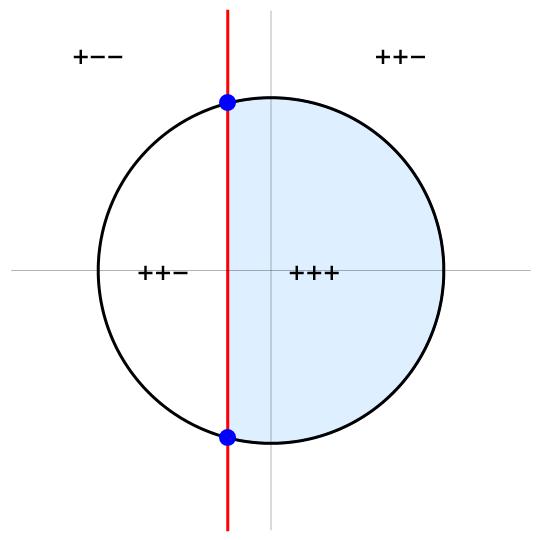}
		\hspace{.5in}
		\includegraphics[height=2in]{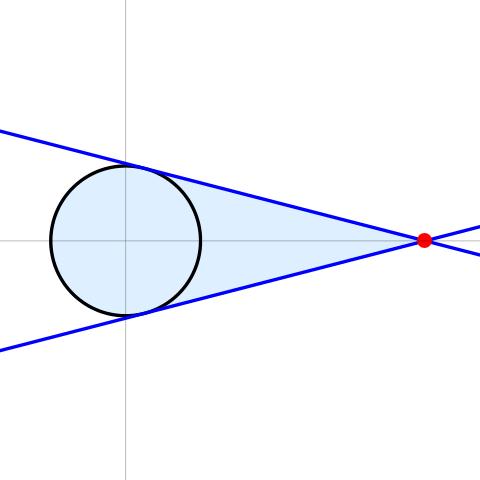}
		\caption{ The curves $\V(f_{A})$ and $\V(g_{A})$ from \Cref{ex:circleandline}}
		. \label{fig:lineandcircle}
	\end{figure}

	\section{Lower dimensional $\Lambda_{k}(A)$}
	\label{sec:Dim01} In this section, we study the geometric and algebraic
	conditions under which $\Lambda_{k}(A)$ is nonempty but has dimension $\leq 1$.
	We do this by considering the convex hull, $\conv(O_{k}(A))$, of the curve $O_{k}
	(A)\subset \R^{3}$ defined in \eqref{eq:Ok(A)}. We will see that when $\Lambda_{k}
	(A)$ has dimension $0$ or $1$, the origin $(0,0,0)$ lies on a face $F$ of
	$\conv(O_{k}(A))$ with $1\leq\dim(F)\leq 2$. We then differentiate when $F$ contains
	\emph{antipodal points} of $O_{k}(A)$ and when it does not. Here, we say that
	$(\lambda_{k}(\theta), -\cos(\theta), -\sin(\theta))$ and
	$(\lambda_{k}(\theta+\pi), -\cos(\theta+\pi), -\sin(\theta+\pi))$ are \emph{antipodal
	points} of $O_{k}(A)$ when $\lambda_{k}(\theta) = \lambda_{n-k+1}(\theta)$, in
	which case,
	\begin{align*}
		(\lambda_{k}(\theta+\pi), -\cos(\theta+\pi), -\sin(\theta+\pi)) & = (-\lambda_{n-k+1}(\theta), \cos(\theta), \sin(\theta)) \\
		                                                                & = (-\lambda_{k}(\theta), \cos(\theta), \sin(\theta)),
	\end{align*}
	giving antipodal points $\pm(\lambda_{k}(\theta), -\cos(\theta), -\sin(\theta))
	.$ This occurs only when the curve $\V(f_{A})$ has a real singularity with a prescribed
	signature. First we consider the dimension of $F$.

	\begin{proposition}
		\label{prop:LowerDimFace} Suppose $O_{k}(A)$ is not contained in a plane through
		the origin and $k\leq (n+1)/2$. Then $\dim(\Lambda_{k}(A))\in \{0,1\}$ if
		and only if $(0,0,0)$ belongs to the relative interior of a face $F$ of
		$\conv(O_{k}(A))$ with $1\leq \dim(F)\leq 2$.
		\begin{itemize}
			\item[(i)] If $\dim(F)=1$, then it is the convex hull of two antipodal
				points on $O_{k}(A)$.

			\item[(ii)] If $\dim(F)=2$ and $F\subseteq \{(t,x,y) : t+ax+by=0\}$, then
				$\Lambda_{k}(A) = \{a+\ii b\}$.
		\end{itemize}
	\end{proposition}
	\begin{proof}
		Recall that the cone over $\Lambda_{k}(A)$,
		$\widehat{\Lambda_k(A)}\subseteq \R^{3}$, has dimension one more than that of
		$\Lambda_{k}(A)$. By \Cref{cor:Duals}, $\widehat{\Lambda_k(A)}$ is the convex
		dual cone of the \emph{conical hull} of the curve $O_{k}(A)$. The cone $\widehat
		{\Lambda_k(A)}$ is full dimensional if and only if its dual cone,
		$\overline{\cone(O_k(A))}$ is pointed, which occurs if and only if $(0,0,0)$
		does not belong to $\conv(O_{k}(A))$. We can therefore assume that $(0,0,0)$
		belongs to $\conv(O_{k}(A))$ and let $F$ be the unique nonempty face of
		$\conv(O_{k}(A))$ containing $(0,0,0)$ in its relative interior.

		Note that $\Lambda_{k}(A) = \emptyset$ is equivalent to
		$\widehat{\Lambda_k(A)}=\{(0,0,0)\}$, which happens if and only if its dual cone,
		$\overline{\cone(O_k(A))}$, is all of $\R^{3}$. This occurs if and only if
		$\conv(O_{k}(A))$ is three-dimensional and $(0,0,0)$ belongs to its interior.
		That is, if $\dim(F)=3$.

		It follows that $\dim(\Lambda_{k}(A)) \in \{0,1\}$ if and only if $(0,0,0)$
		belongs to the boundary of $\conv(O_{k}(A))$. The $0$-dimensional faces of this
		convex body are precisely the points on the curve $O_{k}(A)$. Since
		$(0,0,0)$ does not lie on this curve, it lies on the boundary of $\conv(O_{k}
		(A))$ if and only if it lies in the relative interior of some face $F$ with $1
		\leq \dim(F)\leq 2$.

		($\dim(F)=1$) If $\dim(F)=1$, it is the convex hull of two points
		$(\lambda_{k}(\theta),-\cos (\theta), -\sin(\theta))$ and
		$(\lambda_{k}(\theta'), -\cos (\theta'), -\sin(\theta'))$. Since $(0,0,0)$ is
		a convex combination of these points and the last two coordinates lie on the
		unit circle, these points must be negatives of each other. It follows that $\theta
		' \equiv \theta + \pi \mod 2\pi$ and $\lambda_{k}(\theta+\pi) = -\lambda_{k}(
		\theta)$.

		($\dim(F)=2$) Suppose $\dim(F) = 2$ and $F\subseteq \{(t,x,y) : t+ax+by=0\}$.
		If $a'+\ii b'\in \Lambda_{k}(A)$, then by \eqref{eq:half_plane2}, the linear
		function $t+a'x+b'y$ is nonnegative on $O_{k}(A)$, in which case it must
		vanish identically on $F$. Since $F$ has dimension two, there is a unique
		linear function vanishing on $F$, up to scaling, showing that
		$(a',b') = (a,b)$.
	\end{proof}

	We first explore the case when the face $F$ has no antipodal points. In this
	case, $\Lambda_{k}(A)$ will consist of a single point that corresponds to a tritangent
	line of the curve $\mathcal{V}(f_{A})$ or linear factor of $f_{A}$.

	\begin{definition}
		Let $\ell =t+ax+by$, $f\in \C[t,x,y]$ be homogeneous of degree one and $d$, respectively.
		Let $f^{\rm red}$ denote the square free part of $f$. We say that $\ell$ is \emph{tritangent}
		to $\mathcal{V}(f)$ if the restriction of $f$ to the line $\ell=0$ has
		$\geq 3$ double points. That is,
		$f^{\rm red}(-ax-by,x,y) = (\ell_{1} \cdot \ell_{2} \cdot \ell_{3})^{2} \cdot
		h$
		for some $\ell_{1}, \ell_{2}, \ell_{3}, h\in \C[t,x,y]$.
	\end{definition}

	\begin{theorem}
		\label{thm:tritangent} Suppose $\dim(\Lambda_{k}(A))\in \{0,1\}$ and that $\lambda
		_{k}(\theta) \neq \lambda_{n-k+1}(\theta)$ for all $\theta$. Then
		$\Lambda_{k}(A) = \{a + \ii b\}$ where $t+ax+by$ is either tritangent to
		$\mathcal{V}(f_{A})$ or divides $f_{A}$.
	\end{theorem}
	When the curve $O_{k}(A)$ is smooth, this follows from a characterization of faces
	for convex hulls of smooth algebraic curves in $\R^{3}$ \cite{RanestadSturmfels+2012+157+178}.
	In general, this curve may have singularities.

	\begin{proof}
		Note that the assumption $\lambda_{k}(\theta) \neq \lambda_{n-k+1}(\theta)$ guarantees
		that $O_{k}(A)$ is not contained in a plane through the origin.
		\Cref{prop:LowerDimFace} then implies that $(0,0,0)$ is contained in the
		relative interior of a face $F$ of $\conv(O_{k}(A))$ with $\dim(F)=2$ and
		that $\Lambda_{k}(A) = \{a+\ii b\}$, where $t+ax+by=0$ defines the unique plane
		through the origin containing $F$.

		Let $\ell=t+ax+by$ and suppose that $\ell$ does not divide $f_{A}$. It
		follows that there are only finitely many points of the curve $O_{k}(A)$
		with $\ell=0$. In particular, the vertices of the face $F$ belong to this
		finite intersection. Since $\dim(F)=2$, it has three vertices
		$v_{1}, v_{2}, v_{3}$ (possibly among others), which necessarily belong to $O
		_{k}(A)$. As the curve $O_{k}(A)$ has no antipodal points, the images,
		$[v_{1}], [v_{2}], [v_{3}]$, of these points in $\PP^{2}(\R)$ are all distinct.

		We claim that the restriction of $f_{A}^{\rm red}$ to the line $\ell=0$ has roots
		of multiplicity $\geq 2$ at each of these points. Suppose to the contrary
		that $[v_{i}]$ is a simple root of the restriction of $f_{A}^{\rm red}$ to
		$\ell=0$. It follows that $[v_{i}]$ is not a singular point of $\mathcal{V}(f
		_{A}^{\rm red})$ and that $\mathcal{V}(f_{A}^{\rm red})$ and $\mathcal{V}(\ell
		)$ meet transversely at $[v_{i}]$. From this, we see that $v_{i}$ is a smooth
		point of the curve $O_{k}(A)$ and that $O_{k}(A)$ and the plane $\mathcal{V}(
		\ell)$ meet transversely at $v_{i}$. This implies that the sign of
		$\ell(\lambda_{k}(\theta), -\cos(\theta), -\sin(\theta))$ changes sign at
		$v_{i}$, contradicting the fact that $\ell$ is nonnegative on $O_{k}(A)$.

		Therefore the restriction of $f_{A}^{\rm red}$ to the line $\ell=0$ has $\geq
		3$ roots of multiplicity $\geq 2$ and $\ell$ is tritangent to $\V(f_{A})$.
	\end{proof}

	Up to rescaling by arbitrary real scalars, antipodal points $\pm(\lambda_{k}(\theta
	),-\cos(\theta),-\sin(\theta))$ are arbitrary points $\pm(-\lambda_{k}(p_{1},p_{2}
	),p_{1},p_{2})$ in the variety of $f_{A}$ with $\lambda_{k}(p_{1}, p_{2}) = \lambda
	_{n-k+1}(p_{1}, p_{2})$.

	\begin{lemma}
		\label{lm:restricting singularity} Suppose that $O_{k}(A)$ contains a pair of
		antipodal points. That is, suppose $\lambda_{k}(p_{1},p_{2}) = \lambda_{n-k+1}
		(p_{1},p_{2})$ for some $(p_{1},p_{2})\in \R^{2}\backslash \{(0,0)\}$. Then
		\begin{itemize}
			\item[(i)] $\Lambda_{k}(A)\subseteq \{a+\ii b : \lambda_{k}(p_{1},p_{2}) =
				ap_{1} + bp_{2}\}$, and

			\item[(ii)] if $k < (n+1)/2$, then
				$p = [-\lambda_{k}(p_{1},p_{2}) :p_{1}:p_{2}]$ satisfies $f_{A}(p)=0$
				and $\nabla f_{A}(p)=0$.
		\end{itemize}
	\end{lemma}
	\begin{proof}
		(i) After rescaling, we can take
		$p = (\lambda_{k}(\theta), -\cos(\theta), -\sin(\theta))$. Suppose that $a+\ii
		b\in \Lambda_{k}(A)$. Then the function $\ell(t,x,y) = t+ax+by$ is
		nonnegative on $O_{k}(A)$. In particular, $\ell(p) \geq 0$ and
		$\ell(-p) = -\ell(p)\geq 0$. It follows that
		\[
			0=\ell(p) = \lambda_{k}(\theta) -a\cos(\theta) -b\sin(\theta).
		\]

		(ii) When $k<(n+1)/2$, $k<n-k+1$. By definition for any $k\leq j \leq n-k+1$,
		$\lambda_{n-k+1}(\theta) \leq \lambda_{j}(\theta) \leq \lambda_{k}(\theta)$.
		Since $\lambda_{n-k+1}(\theta) =\lambda_{k}(\theta)$ we see that these are
		equalities. It follows that the matrix $I_{n}\lambda_{k}(\theta) - \cos(\theta
		) \Re A - \sin(\theta) \Im A$ has corank $\geq n-2k+2 \geq 2$. By
		\Cref{lm:singDegree}, $f_{A}$ has multiplicity $\geq 2$ at $[\lambda_{k}(\theta
		) :-\cos(\theta):-\sin(\theta)]$.
	\end{proof}

	\begin{corollary}
		\label{cor:SingRestriction} Let $S =\{[p_{0}:p_{1}:p_{2}]\in \mathcal{V}(f_{A}
		) : -p_{0}=\lambda_{k}(p_{1}, p_{2}) = \lambda_{n-k+1}(p_{1}, p_{2})\}$ and consider
		the projective linear space $V\subset \PP^{2}(\R)$ spanned by $S$. Then
		\[
			\dim(\Lambda_{k}(A))\leq 1-\dim(V).
		\]
		In particular, if $\dim(V)=2$ then $\Lambda_{k}(A)$ is empty and if $\dim(V)=
		1$ then $\Lambda_{k}(A)$ is either empty or the single point $\{a+\ii b : [1:
		a:b]\in V^{\perp}\}$.
	\end{corollary}

	Note that $[1:a:b]\in V^{\perp}$ if and only if $t+ax+by=0$ for all $[t:x:y]\in
	V$.

	\begin{remark}
		The condition $\lambda_{k}(p_{1}, p_{2}) = \lambda_{n-k+1}(p_{1},p_{2})$ is well-defined
		on points in $\PP^{2}(\R)$. It is clearly invariant under nonnegative
		scaling. Moreover, by definition
		$\lambda_{k}(-p_{1}, -p_{2})= -\lambda_{n-k+1}(p_{1},p_{2})$. Similarly,
		$\lambda_{n-k+1}(-p_{1},-p_{2})=-\lambda_{k}(p_{1}, p_{2})$.
	\end{remark}

	When $\Lambda_{k}(A)$ is one-dimensional, we see from \Cref{prop:LowerDimFace}
	and \Cref{lm:restricting singularity} that the curve $\V(f_{A})$ has a point $p$
	of multiplicity $\geq 2$. In this case, we would like to determine the end points
	of the line segment $\Lambda_{k}(A)$. These will correspond to lines passing
	through $p$ with additional tangency to $\mathcal{V}(f_{A})$.

	\begin{definition}
		\label{def:y-tangent} Let $p\in \mathcal{V}(f_{A})$. Let $f^{\rm red}_{A}$
		denote the square-free part of the factorization of $f_{A}$ into irreducible
		polynomials. Then $p$ is a point of $\mathcal{V}(f^{\rm red}_{A})$ of some multiplicity
		$d\geq 1$. Recall from \Cref{lm:singDegree} that for any $q\in \C^{3}$ the
		restriction $f^{\rm red}_{A}(rq+sp)$ has a factor of $r^{d}$. We say that the
		line $\{ [rq + sp]: [r:s] \in \mathbb{P}^{1}(\mathbb{C}) \}$ is $p$-\emph{tangent}
		to $\mathcal{V}(f_{A})$ if one of the following holds:
		\begin{itemize}
			\item[(I)] $f^{\rm red}_{A}(rq + sp) = r^{d}\cdot \ell^{2}\cdot h$ for some
				$\ell\in\C[r,s]_{1}$ and $h\in \C[r,s]$, or

			\item[(II)] $f^{\rm red}_{A}(rq + sp) = r^{d+1}\cdot h$ for some $h\in \C[r
				,s]$.
		\end{itemize}
	\end{definition}

	Observe that part (I) in this definition allows for $\ell=r$, which satisfies
	(II). Similarly, both conditions are satisfied when the restriction
	$f^{\rm red}_{A}(rq + sp)$ is identically zero, in which case the linear form defining
	the line is a factor of $f^{\rm red}_{A}$ and hence $f_{A}$.

	\begin{theorem}
		\label{thm:ptangents} If $k<(n+1)/2$ and $p=[p_{0}:p_{1}:p_{2}]\in\mathcal{V}
		(f_{A})$ with $-p_{0}=\lambda_{k}(p_{1},p_{2}) = \lambda_{n-k+1}(p_{1},p_{2})$
		then $\Lambda_{k}(A)$ is the convex hull of its elements $a+\ii b$ for which
		$t+ax+by$ is $p$-tangent. That is,
		\[
			\Lambda_{k}(A) = \conv(\{a+ \ii b \in \Lambda_{k}(A): t + ax+ by = 0 \text{
			is }p\text{-tangent to }\mathcal{V}(f_{A})\}).
		\]
	\end{theorem}

	We will often perform a change of coordinates to take $p$ to the point
	$[0:0:1]$. The following will be a useful way of discovering $p$-tangent lines.

	\begin{lemma}
		\label{lm:limittangent} Suppose that $[0:0:1]$ is a point of $\mathcal{V}(f_{A}
		)$ of multiplicity $d\geq 1$. Then $t + cx$ is $[0:0:1]$-tangent to
		$\mathcal{V}(f_{A})$ of type (II) if and only if $t+cx$ divides the lowest
		degree part of $f^{\rm red}_{A}(t,x,1)$. Moreover, this condition is
		satisfied when $\lim_{x \to 0^+ }\lambda_{k}(x, 1)/x = c$ or
		$\lim_{x \to 0^-}\lambda_{k}(x, 1)/x = c$ for some $k$.
	\end{lemma}
	\begin{proof}
		Recall the notation $\lambda_{k}(x,1):= \lambda_{k}(x\Re A + \Im A )$. For
		ease of notation let $f$ denote $f_{A}^{\rm red}$. By assumption $[0:0:1]$ is
		a point of $\mathcal{V}(f)$ of some multiplicity $d\geq 1$. Consider the
		expansion $f(t,x,y) =y^{m-d}h_{d}+y^{m-d-1}h_{d+1}+\hdots + h_{m}$, where $h_{j}
		\in \R[t,x]$ is homogeneous of degree $j$ and $m = \deg(f)$.

		To see the first equivalence, note that the line $t+cx=0$ is spanned by
		$q =(-c,1,0)$ and $p = (0,0,1)$. We expand $f(rq+sp) = \sum_{j=d}^{m}s^{m-j}r
		^{j}h_{j}(-c,1)$ to see that $h_{d}(-c,1)=0$ if and only if $r^{d+1}$ divides
		$f(rq+sp)$. Since $h_{d}(t,x)$ is homogeneous and bivariate, $t+cx$ divides $h
		_{d}(t,x)$ if and only if $h_{d}(-c,1)=0.$

		For the further claim, we find the limit
		\[
			\lim_{x \to 0^{\pm}}\frac{f(-\lambda_{k}(x,1),x,1)}{x^{d}}= 0
		\]
		by noting that its numerator is identically zero. We can then rewrite this
		limit as
		\begin{eqnarray*}
			0& = &\lim_{x \to 0^{\pm}} \frac{h_{d}(-\lambda_{k}(x,1),x)}{x^{d}} + \frac{h_{d+1}(-\lambda_{k}(x,1),x)}{x^{d}}
			+ \cdots + \frac{h_{m}(-\lambda_{k}(x,1),x)}{x^{d}}\\
			& = & \lim_{x \to 0^{\pm}} h_d\left(-\frac{\lambda_{k}(x,1)}{x},1\right) +
			xh_{d+1}\left(-\frac{\lambda_{k}(x,1)}{x},1\right) + \cdots + x^{m-d}h_{m}\left(-\frac{\lambda_{k}(x,1)}{x},1\right)\\
			& = & h_d(-c,1).
		\end{eqnarray*}
		In the last limit, note that the $\lim_{x \to 0^{\pm}}h_{j}(-\frac{\lambda_{k}(x,1)}{x}
		,1) = h_{j}(-c,1)$ for all $j$. This shows that $t + cx$ must be a linear
		factor of $h_{d}(t,x)$.
	\end{proof}

	\begin{lemma}
		\label{lem:semialgFunction} Let $\varphi:\R\to\R$ be a continuous semialgebraic
		function. If $\inf\{\varphi(y): y\in \R\}$ equals zero, then either $\varphi(
		y^{*})=0$ for some $y^{*}\in \R$, $\lim_{y \to \infty}\varphi(y)= 0$, or $\lim
		_{y \to -\infty}\varphi(y) = 0$.
	\end{lemma}

	\begin{proof}
		By definition, there is some polynomial $F(y,z)\in \R[y,z]$ for which
		$F(y,\varphi(y))=0$ for all $y\in \R$. Then $\varphi$ is differentiable at all
		but finitely many points and has finitely many critical points, at which
		$\varphi'$ is zero or undefined. See, for example, \cite[Thm 5.56]{Basu2006}.
		Let $[a,b]$ be a closed interval containing those critical points. The minimum
		of $\varphi(y)$ over $y\in [a,b]$ is attained. If this minimum is zero, then
		there is a point $y^{*}\in [a,b]$ with $\varphi(y^{*})=0$. If not, then the infimum
		of $\varphi$ over either $(-\infty, a)$ or $(b, \infty)$ must be zero. On
		$(-\infty, a)$ and $(b, \infty)$ the function $\varphi$ is monotone. It follows
		that either $\lim_{y \to \infty}\varphi(y)$ or $\lim_{y \to -\infty}\varphi(y
		)$ equals zero.
	\end{proof}

	\begin{proof}[Proof of \Cref{thm:ptangents}]
		After a change of coordinates, we may assume that $p=[0:0:1]$. From \Cref{lm:restricting
		singularity}, it follows that $\Lambda_{k}(A)\subset \R$. If $\Lambda_{k}(A)$
		is empty then there is nothing to prove. Therefore, we assume that $\Lambda_{k}
		(A)$ is non-empty. By applying another affine linear transformation we can
		assume that $\Lambda_{k}(A) = [-1,0]$ or $\Lambda_{k}(A) =\{0\}$. In either case
		$0\in \Lambda_{k}(A)$ and $\varepsilon\not\in \Lambda_{k}(A)$ for all $\varepsilon
		\in \R_{>0}$. It suffices to show that $t=0$ is a $p$-tangent line of
		$\mathcal{V}(f_{A})$. By definition, if $t$ is a factor of $f_{A}$, then,
		since it passes through $p$, it is $p$-tangent. Therefore we can assume that
		$t$ is not a factor of $f_{A}$.

		Recall the notation $\lambda_{j}(x,y) = \lambda_{j}( x\Re(A) + y\Im(A) )$. Since
		$0\in\Lambda_{k}(A)$, by \autoref{thm:MembershipTest}(a)(f), we have that for
		all $y\in \R$,
		\begin{equation}
			\label{eq:yhatbounds}\lambda_{n-k+1}( 1, y) \leq 0 \leq \lambda_{k}( 1,y ).
		\end{equation}
		Similarly, since $\varepsilon \not\in \Lambda_{k}(A)$, there exists
		$y_{\varepsilon}\in \R$ for which
		$\lambda_{k}( 1,y_{\varepsilon})<\varepsilon$ or
		$\lambda_{n-k+1}( 1, y_{\varepsilon}) > \varepsilon$. Since $\lambda_{n-k+1}(
		1,y)\leq 0$ for all $y\in \R$, it must be that
		$\lambda_{k}( 1,y_{\varepsilon})<\varepsilon$. It follows that
		\begin{equation}
			\label{eq:inf}\inf_{y \in \mathbb{R}}\lambda_{k}(1,y) = 0.
		\end{equation}

		Because the function $y\mapsto \lambda_{k}(1,y)$ is continuous and
		semialgebraic, it follows from \Cref{lem:semialgFunction} that either $\lambda
		_{k}(1, y^{*}) = 0$ for some $y^{*}\in \R$,
		$\lim_{y \to \infty}\lambda_{k}(1,y) = 0$, or $\lim_{y \to -\infty}\lambda_{k}
		(1,y) = 0$.

		For $x>0$, $\lambda_{k}(1,1/x) = \lambda_{k}(x,1)/x$. Then $\lim_{y \to
		\infty}\lambda_{k}(1,y) = \lim_{x\to 0^+}\lambda_{k}(x,1)/x$. Similarly, for
		$x<0$, $\lambda_{k}(1,1/x) = \lambda_{n-k+1}(x,1)/x$. So $\lim_{y \to -\infty}
		\lambda_{k}(1,y) = \lim_{x\to 0^-}\lambda_{n-k+1}(x,1)/x$. If either of
		these limits equal zero, \Cref{lm:limittangent} gives that $t=0$ is a $p$-tangent
		line of $\mathcal{V}(f_{A})$.

		Finally, suppose $\lambda_{k}(1, y^{*}) = 0$ for some $y^{*}\in \R$. If the point
		$(0, 1, y^{*})$ is a singularity of $\mathcal{V}(f^{\rm red}_{A})$, then the
		restriction $f^{\rm red}_{A}(0,1,y)$ has a zero of multiplicity $\geq 2$ at
		$y^{*}$, by \Cref{lm:singDegree}.

		In the case that this point is non-singular we can consider the function $\lambda
		_{k}(x,y)$. This function is locally minimized at $(1, y^{*})$. Therefore,
		it must be the case that its gradient is zero. Hence $\nabla f^{\rm red}(0,1,
		y^{*})$ is proportional to $(1,0,0)$. See \Cref{lem:lambda_diff}. From here we
		can conclude that the line $t = 0$ is tangent to $\mathcal{V}(f^{\rm red}_{A}
		)$ at $(0, 1, y^{*})$. In particular, the restriction $f^{\rm red}_{A}(0,1,y)$
		is either identically zero or has a root of multiplicity $\geq 2$. In either
		case, $t=0$ is a $p$-tangent of $\mathcal{V}(f_{A})$.
	\end{proof}

	\begin{example}
		\label{ex:quarticPtangent} Consider the Kippenhahn polynomial
		\[
			f_{A} = t^{4} - \frac{13}{36}t^{2} x^{2} + \frac{1}{36}x^{4} - \frac{1}{4}t
			^{2} y^{2} - \frac{2}{25}t x y^{2} + \frac{1}{100}x^{2} y^{2}.
		\]
		The curve $\mathcal{V}(f_{A})$ is singular at the point $p = [0:0:1]$. The
		lowest degree part of $f_{A}(t,x,1)$ is
		$h_{2}(t,x) = \frac{1}{4}t^{2} - \frac{2}{25}t x + \frac{1}{100}x^{2} = (t +
		\frac{1}{25}(4 + \sqrt{41})x)(t + \frac{1}{25}(4 - \sqrt{41})x)$. Both factors
		give $p$-tangent lines of $f_{A}$. One can check that
		$a = \frac{1}{25}(4 - \sqrt{41})$ belongs to $\Lambda_{2}(A)$ and will form
		one of its end points. The line $t + \frac{1}{3}x=0$ passes through $[0:0:1]$
		and is additionally tangent to $\mathcal{V}(f_{A})$ at the point $[1:-3:0]$,
		showing that it is also $p$-tangent to $\mathcal{V}(f_{A})$. In this case,
		$\Lambda_{2}(A)$ is the line segment $[\frac{1}{25}(4 - \sqrt{41}), 1/3]$.

		The curve $O_{2}(A)$ along with a portion of the boundary of its convex hull
		is shown in \Cref{fig:nodalsingularityquartic}. The plane $t + \frac{1}{3}x=0$
		defines a two-dimensional face of this convex hull, whereas
		$t + \frac{1}{25}(4 - \sqrt{41})x=0$ is a limit of supporting hyperplanes of
		edges.
	\end{example}

	\begin{figure}
		\includegraphics[height=2in]{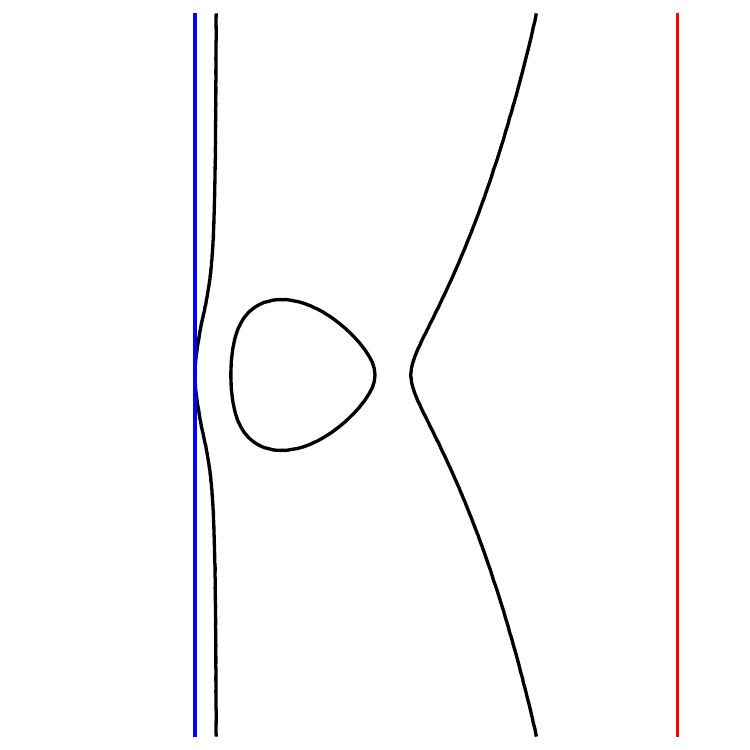}
		\includegraphics[height=2in]{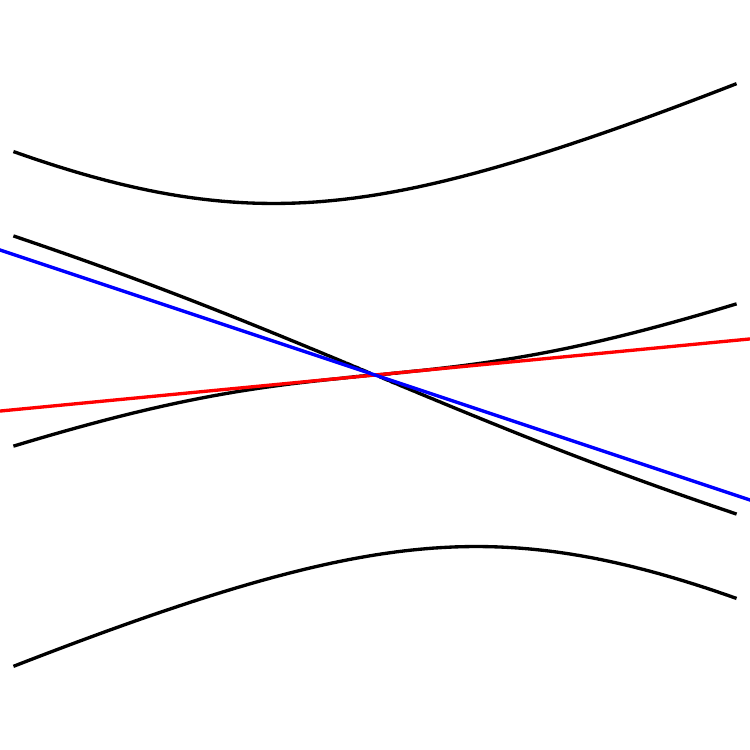}
		\includegraphics[height=2in]{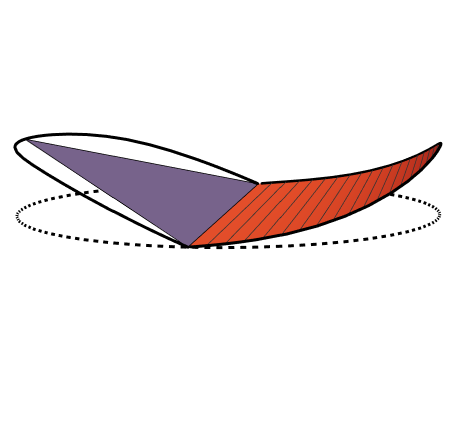}
		\caption{The curve $\mathcal{V}(f_{A})$ and two $[0:0:1]$-tangent lines from
		\Cref{ex:quarticPtangent} in the $\{t=1\}$ and $\{y=1\}$ affine charts and
		curve $O_{2}(A)$.}
		\label{fig:nodalsingularityquartic}
	\end{figure}

	Lastly, we consider the case when $k \geq (n+1)/2$. It is known that if $\Lambda
	_{k}(A) = \{ a + \ii b \}$ is nonempty then $a + \ii b$ must be an eigenvalue of
	$A$ with geometric multiplicity $\geq 2k-n$ \cite[Proposition 2.2]{CHOI2006828}.
	Here we give an algebraic analogue of this result.

	\begin{lemma}
		\label{lm:linearFactors} If $k \geq (n+1)/2$ and $a+ \ii b \in \Lambda_{k}(A)$
		then $(t + ax + by)^{2k - n}$ divides the Kippenhahn polynomial $\det(tI_{n}
		+ x \Re A + y \Im A )$.
	\end{lemma}

	\begin{proof}
		Let $\ell = t + ax+ by$. By the Membership test \autoref{thm:MembershipTest},
		for all real points in the line $\mathcal{V}_{\R}(\ell)$, the matrix
		$M = tI_{n} + x \Re A + y \Im A$ has at most $n-k$ strictly positive
		eigenvalues and at most $n-k$ strictly negative eigenvalues. It follows the rank
		of $M$ at this point is at most $2(n-k)$, meaning that its $d\times d$
		minors are all zero for $d=2(n-k)+1$. Since this holds for every point $(t,x,
		y)\in \mathcal{V}_{\R}(\ell)$ and $\ell$ is irreducible, $\ell$ divides the
		$d\times d$ minors of $M$. By \Cref{lem:detDivisors} below, it follows that
		$\ell^{n-d+1}= \ell^{2k-n}$ divides $f_{A} =\det(M)$.
	\end{proof}

	The converse of this lemma is not true even for the case $n = 3$ and $k = 2$. For
	example, one could take $A$ to be a $3\times 3$ diagonal matrix whose diagonal
	entries are not co-linear in $\C$. One can check that in this case $f_{A}$ has
	three linear factors but that $\Lambda_{2}(A)$ is empty.

	\begin{lemma}
		\label{lem:detDivisors} Let $M$ be an $n\times n$ matrix with entries in a unique
		factorization domain $R$. Let $\ell\in R$ be irreducible and suppose that
		$\ell$ divides all $d\times d$ minors of $M$, where $d\leq n$. Then
		$\ell^{n-d+1}$ divides $\det(M)$.
	\end{lemma}
	\begin{proof}
		We induct on $n-d$ and follow the proof of \cite[Lemma 4.7]{PLAUMANN201348}.
		For $n-d=0$, this holds by assumption. Suppose that $n-d\geq 1$ and that
		$\ell$ divides all $d\times d$ minors of $M$. This also holds for any of the
		$d\times d$ minors of any $(n-1)\times (n-1)$ submatrix of $M$, so, by induction
		$\ell^{n-d}$ divides the $(n-1)\times (n-1)$ minors of $M$. We use the
		identity $M^{\rm adj}M = \det(M)I_{n}$, where $M^{\rm adj}$ denotes the adjugate
		matrix of $M$, whose entries are signed $(n-1)\times (n-1)$ minors of $M$. Taking
		the determinant of boths sides, we find that
		$\det(M^{\rm adj})= \det(M)^{n-1}$.

		Suppose that $\det(M) = \ell^{m}h$ where $\ell$ does not divide $h$. Then $\det
		(M^{\rm adj}) = \ell^{m(n-1)}h^{n-1}$. By assumption, $\ell^{n-d}$ divides all
		entries of $M^{\rm adj}$ and so $\ell^{n(n-d)}$ divides its determinant.
		Since $\ell$ is irreducible, it follows that $\ell^{n(n-d)}$ must divide $\ell
		^{m(n-1)}$, giving $n(n-d)\leq m(n-1)$. If $m\leq n-d$, then
		$m(n-1)\leq (n-d)(n-1)< (n-d)n$. Since this contradicts the inequality above,
		we conclude that $n-d+1\leq m$, as desired.
	\end{proof}

	\section{Algorithms}
	\label{sec:algorithm}

	In this section we discuss an algorithm to compute the rank-$k$ numerical range
	of an $n\times n$ matrix. This builds off \Cref{alg:membership}, which tests membership
	of a given point in $\Lambda_{k}(A)$, as well as the characterizations of the
	boundary of $\Lambda_{k}(A)$ in \Cref{sec:algBoundary} and of algebraic
	conditions satisfied when $\Lambda_{k}(A)$ is lower dimensional in
	\Cref{sec:Dim01}. To build up to the full algorithm, we first establish some
	subroutines to compute $p$-tangents and tritangent lines of
	$\mathcal{V}(f_{A})$, as defined in \Cref{sec:Dim01}.

	Many of these computations involve the discriminant. Formally, given a polynomial
	$p(t) = \sum_{k=0}^{d}a_{k}t^{k}$, the discriminant is a polynomial
	${\rm Disc}_{t}(p)$ in $\mathbb{Q}[a_{0}, \hdots, a_{d}]$. It is the unique minimal
	polynomial (up to scaling) satisfying ${\rm Disc}_{t}(p)=0$ whenever $p(t)$ that
	has a double root. For a multivariate polynomial
	$f\in \C[x_{1}, \hdots, x_{n}]$, we use ${\rm Disc}_{x_i}(f)$ to denote the discriminant
	of $f$ interpreted as a univariate polynomial in $x_{i}$. This is a polynomial
	in $\C[x_{j}:j\neq i]$. For a detailed discussion on the discriminant and its properties
	see \cite[Section 4.1]{Basu2006}.

	\begin{algorithm}
		[H]
		\caption{Singularity Tangents
		\smallskip
		\\
		\textbf{Input:} $A\in \C^{n\times n}$ and $p\in \V(f_{A})$ \\
		\textbf{Output:}
		$\mathcal{T}= \{(a,b)\in \R^{2}: t+ax+by \text{ is }p\text{-tangent to }\mathcal{V}
		(f_{A}) \}$.
		\smallskip
		}
		\label{alg:SingTangents}
		\begin{algorithmic}
			[0] \State Take an invertible $L:\R^{2}\to \R^{2}$ as in \eqref{eq:coordChange}
			with $[u_{02}:u_{12}:u_{22}]=p$. \State Let $f = f_{L\cdot A}$ and let $f^{\rm
			red}$ be the square-free part of $f$. \State Let $h_{d}(t,x)$ to be the
			smallest degree part of $f^{\rm red}(t,x,1)$. \State Take $\mathcal{S}= \{c
			\in \R : h_{d}(-c,1)=0\}\cup \{c\in \R :{\rm Disc}_{y}(f^{\rm red})(-c,1)=0
			\}$ \State \Return $\mathcal{T}= \{L^{-1}(c,0) : c\in \mathcal{S}\}$.
		\end{algorithmic}
	\end{algorithm}

	\begin{lemma}
		The output of \Cref{alg:SingTangents} is the set of $(a,b)\in \R^{2}$ for which
		$t+ax+by$ is $p$-tangent to $\mathcal{V}(f_{A})$.
	\end{lemma}
	\begin{proof}
		Note that $t+ax+by$ is a $p$-tangent line of $f_{A}$ if and only if $t+cx$ is
		a $[0:0:1]$-tangent line of $f_{L\cdot A}$, where $(c,0) = L(a,b)$. Therefore,
		it suffices to show that the computed sets $\mathcal{S}$ is the set of
		$c\in \R$ for which $t+cx$ is $[0:0:1]$-tangent to
		$\mathcal{V}(f_{L\cdot A})$. By \Cref{lm:limittangent}, $t+cx$ is a
		$[0:0:1]$-tangent line of the form (II) if and only if $h_{d}(-c,1)=0$.

		We claim that $t+cx$ is a $[0:0:1]$-tangent line of the form (I) with
		$\ell\neq r$ if and only if ${\rm Disc}_{y}(f^{\rm red})(-c,1)=0$. As above,
		the line $t+cx=0$ is parametrized by the points $(-c,1,0)$ and $(0,0,1)$. In
		particular, $f^{\rm red}(r(-c,1,0)+s(0,0,1))=f^{\rm red}(-cr,r,s)$ has a
		square factor $\ell^{2}\neq r^{2}$ in $\C[r,s]$ if and only if its
		restriction to $r=1$ has a repeated root. This happens if and only if its discriminant,
		${\rm Disc}_{y}(f^{\rm red})(-c,1)$, is zero.
	\end{proof}

	As discussed in \Cref{cor:SingRestriction}, the presence of antipodal points
	gives a restriction on higher-rank numerical ranges and guaranties that its
	dimension is at most 1. Given our definition of antipodal points it is
	possible to have infinitely many of them, in particular, when the Kippenhanh
	curve has irreducible components of multiplicity $\geq 2$. However, we show
	that it is only necessary to consider antipodal points that are also
	singularities of $\mathcal{V}(f_{A})$.

	\begin{lemma}
		\label{lem:SvS'} Suppose $\lambda_{k}(p_{1}, p_{2}) = \lambda_{n-k+1}(p_{1},
		p_{2})$ for some $(p_{1}, p_{2})\in \R^{2}\backslash\{(0,0)\}$ and $k<(n+1)/2$.
		Then one of the following holds:
		\begin{itemize}
			\item[(i)] $[-\lambda_{k}(p_{1}, p_{2}):p_{1}:p_{2}]$ is a singularity of $\mathcal{V}
				(f^{\rm red}_{A})$,

			\item[(ii)] $\Lambda_{k}(A)= \{a+\ii b\}$ for some $a,b\in \R$ where $(t+ax
				+by)^{2}$ divides $f_{A}$, or

			\item[(iii)] $\Lambda_{k}(A)$ is empty.
		\end{itemize}
	\end{lemma}
	\begin{proof}
		Let $\ell = t + ax + by$. Suppose that (i) and (iii) are false. That is,
		there exists some point $a+\ii b \in \Lambda_{k}(A)$ and also
		$[-\lambda_{k}(p_{1},p_{2}):p_{1}:p_{2}]$ is a smooth point of
		$\mathcal{V}(f_{A}^{\rm red})$. In particular,
		$\frac{\partial f_{A}^{\rm red}}{\partial t}\not = 0$ by
		\cite[Lemma 2.4]{PLAUMANN201348}. By the implicit function theorem there exists
		a neighborhood $(p_{1},p_{2}) \in U \subseteq \mathbb{R}^{2}$ and a unique continuously
		differentiable function $\varphi(x,y)$ such that
		\[
			f_{A}^{\rm red}(-\varphi(x,y),x,y)=0 \text{ for all }(x,y) \in U.
		\]
		Since $\lambda_{k}, \ldots, \lambda_{n-k}$ are locally continuously differentiable
		functions that satisfy the previous statement they must all agree on $U$.
		Thus, for every $(q_{1},q_{2})$ in $U$, the point
		$[-\lambda_{k}(q_{1},q_{2}):q_{1}:q_{2}]$ is a smooth point of
		$\mathcal{V}(f_{A}^{\rm red})$ and
		$\lambda_{k}(q_{1},q_{2}) = \lambda_{n-k+1}(q_{1},q_{2})$. Following the arguments
		of \autoref{lm:restricting singularity} we can conclude that $-\lambda_{k}(q_{1}
		,q_{2}) +aq_{1} +bq_{2} = 0$. Since there are infinitely many points
		$(t,x,y$ in $\mathcal{V}(\ell)$ at which the matrix
		$tI_{n} + x\Re A + y \Im A$ has corank $\geq 2$, it follows that $\ell$ must
		divide all of the $(n-1) \times (n -1)$ minors by Bezout's theorem. Then, by
		\Cref{lem:detDivisors}, $\ell^{n - (n + 1) + 1}= \ell^{2}$ divides $f$.
	\end{proof}

	As shown in \Cref{lm:linearFactors}, when $k \geq (n+1)/2$, we only need to check
	points corresponding to linear factors with a multiplicity of $2k-n$. For the
	sake of completeness we present a simple algorithm that computes
	$\Lambda_{k}(A)$ under this assumption.

	\begin{algorithm}
		[H]
		\caption{Computing the rank-$k$ numerical range of a matrix for
		$(n+1)/2\leq k\leq n$
		\smallskip
		\\
		\textbf{Input:} $A\in \C^{n\times n}$, $k\in \mathbb{Z}_{\geq 0}$ with
		$(n+1)/2\leq k\leq n$ \\
		\textbf{Output:} $d \in \{-1,0\}$ : dimension of $\Lambda_{k}(A)$ \\
		If $d = -1$, return the empty set.\\
		If $d = 0$, returns $(a,b) \in \mathbb{R}^{2}$ such that $\Lambda_{k}(A) = \{
		a + bi\}$.
		\smallskip
		}
		\label{alg:main2}
		\begin{algorithmic}
			[0] \State $f_{A} = \det(t I_{n} + x \Re A + y \Im A)$. \State Compute all
			distinct linear factors of $f_{A}$ (or all factors appearing with power $\geq
			2k-n$) and perform membership test on each.
		\end{algorithmic}
	\end{algorithm}

	At this point we present the main algorithm for computing rank-$k$ numerical ranges.
	Roughly speaking, this algorithm splits the problem in different cases given by
	the presence of antipodal points and the possible dimension of $\Lambda_{k}(A)$.
	Then, for each of these cases the problem is reduced to applying the membership
	test to a finite list of points to determine which of those points belong to $\Lambda
	_{k}(A)$.

	\begin{algorithm}
		[H]
		\caption{Computing the rank-$k$ numerical range of a matrix for $k< (n+1)/2$
		\smallskip
		\\
		\textbf{Input:} $A\in \C^{n\times n}$, $k\in \mathbb{Z}$ with
		$1\leq k< (n+1)/2$ \\
		\textbf{Output:} ${\rm dim}\in \{-1,0,1,2\}$ : dimension of $\Lambda_{k}(A)$
		\\
		If ${\rm dim}= 0$, returns $a+ \ii b\in \C$ such that
		$\Lambda_{k}(A) = \{ a + \ii b\}$.\\
		If ${\rm dim}= 1$, returns $a+ \ii b, c+ \ii d \in \C$ such that $\Lambda_{k}
		(A) ={\rm conv}\{a+ \ii b, c+ \ii d\}$.\\
		If ${\rm dim}= 2$, returns a polynomial $g_{A}$ vanishing on the boundary of
		$\Lambda_{k}(A)$ and a representative for each connected component of
		$\C\backslash \{a+\ii b : g_{A}(a,b)=0\}$ that is contained in
		$\Lambda_{k}(A)$.
		\smallskip
		}
		\label{alg:main}
		\begin{algorithmic}
			[0] \State Take $f_{A} = \det(t I_{n} + x \Re A + y \Im A)$ and let $f_{A}^{\rm
			red}$ be the square-free part of $f_{A}$ \State Compute $S= \{ [p_{0}:p_{1}
			:p_{2}] \in \Sing_{\R}(f_{A}^{\rm red}) : -p_{0}= \lambda_{k}(p_{1},p_{2})
			= \lambda_{n-k+1}(p_{1},p_{2})\}$ \State Compute $V= \Span S \subseteq \Pj^{2}$
			and $\dim V\in \{-1,0,1,2\}$ \If{$\dim V=2$} \Return{$\dim=-1$} \EndIf \If{$\dim V=1$}
			compute $\{[1:a:b]\}=V^{\perp}$ \State Use \Cref{alg:membership} to test
			$a+\ii b \in \Lambda_{k}(A)$. \If{$a+\ii b \in \Lambda_{k}(A)$} \Return{$\dim=0$, $\{a+ \ii b\}$}
			\Else \hspace{.05in}\Return{$\dim=-1$} \EndIf \EndIf \If{$\dim V=0$}
			$S = \{p\}$ \State Use \Cref{alg:SingTangents} to compute
			$\mathcal{T}= \{a+\ii b : t+ax+by \text{ is $p$-tangent to }\mathcal{V}(f_{A}
			)\}$
			\State Use \Cref{alg:membership} to compute
			$\mathcal{T}\cap \Lambda_{k}(A)$. \If{$|\mathcal{T}\cap \Lambda_{k}(A)|\geq 2$}
			compute end points $\{a+\ii b, c+\ii d\}$ of
			$\conv(\mathcal{T}\cap \Lambda_{k}(A))$ \State \Return{$\dim=1$, $\{a+\ii b, c+\ii d\}$}
			\EndIf \If{$\mathcal{T}\cap \Lambda_{k}(A)=\{a+\ii b\}$} \Return{$\dim=0$, $\{a+\ii b\}$}
			\EndIf \If{$\mathcal{T}\cap \Lambda_{k}(A)=\emptyset$} \Return{$\dim=-1$} \EndIf
			\EndIf \If{$\dim V = -1$} compute $g_{A}\in \R[a,b]$ as in \Cref{def:g_A}
			\State Compute representatives $\mathcal{R}$ for each bounded connected
			component of $\C \backslash \{a+\ii b : g_{A}(a,b)=0\}$ \State Use
			\Cref{alg:membership} to compute $\mathcal{R}\cap \Lambda_{k}(A)$. \If{$\mathcal{R}\cap \Lambda_{k}(A)\neq \emptyset$}
			\Return{$\dim=2$, $g_{A}$, $\mathcal{R}\cap \Lambda_{k}(A)$} \Else \hspace{.05in}Compute
			$\mathcal{C}= \{a+\ii b : t+ax+by\text{ divides $f_{A}$ or is tritangent to
			$f_{A}^{\rm red}$}\}$
			(Remark~\ref{rem:tritangents}). \State Use \Cref{alg:membership} to compute
			$\mathcal{C}\cap \Lambda_{k}(A)$. \If{$\mathcal{C}\cap \Lambda_{k}(A)=\{a+\ii b\}$}
			\Return{$\dim=0$, $\{a+\ii b\}$} \Else \hspace{.05in}\Return{$\dim=-1$}
			\EndIf \EndIf \EndIf
		\end{algorithmic}
	\end{algorithm}

	\begin{theorem}
		Let $A\in \C^{n \times n}$ for which $f_{A}$ is not a power of a linear form.
		For any $k< (n+1)/2$, \Cref{alg:main} computes the dimension and stated description
		of $\Lambda_{k}(A)$.
	\end{theorem}

	\begin{proof}
		Let $S$ and $V$ be as computed in the algorithm. Note that $V$ is a
		projective linear subspace of $\PP^{2}$, so $\dim(V)\in \{-1,0,1,2\}$. Let
		\[
			S' = \{[p_{0}:p_{1}:p_{2}]\in \PP^{2}(\R):-p_{0} = \lambda_{k}(p_{1}, p_{2}
			) = \lambda_{n-k+1}(p_{1}, p_{2}) \} \text{ and }V' ={\rm span}(S') \subseteq
			\PP^{2}.
		\]
		The set $S$ is contained in $S'$ and so $\dim(V')\geq \dim(V)$. By \Cref{cor:SingRestriction},
		if $\dim(V)=2$, then $\Lambda_{k}(A)$ is empty and if $\dim(V)=1$, then
		$\Lambda_{k}(A)$ is either empty or the single point $a+\ii b$ where
		$[1:a:b]=V^{\perp}$, which can be determined by the membership test. If
		$\dim(V)=0$, then $S$ is a single point $p$. By
		\Cref{lm:restricting singularity}(i), $\Lambda_{k}(A)$ is contained in the
		line $\{a+\ii b : p_{0}+ap_{1}+bp_{2}=0\}$. By \Cref{thm:ptangents},
		$\Lambda_{k}(A)$ is the convex hull of the set of
		$a+\ii b \in \mathcal{T}\cap \Lambda_{k}(A)$.

		It remains to examine the case $S=\emptyset$. If $S'$ is nonempty, then by
		\Cref{lem:SvS'}, $\Lambda_{k}(A)$ is either empty or equal to $\{a+\ii b\}$ where
		$t+ax+by$ is a linear factor of $f_{A}$, in which case
		$a+\ii b\in \mathcal{C}$. Otherwise, $S'$ is empty. By \Cref{thm:tritangent},
		$\Lambda_{k}(A)$ is then either two-dimensional, a singleton $\{a+\ii b\}$ with
		$a+\ii b \in \mathcal{C}$, or empty.

		We claim that $\Lambda_{k}(A)$ is two-dimensional if and only if
		$\mathcal{R}\cap \Lambda_{k}(A)$ is nonempty. By \autoref{thm:boundary}, the
		boundary of $\Lambda_{k}(A)$ is contained in
		$\mathcal{V}_{\mathbb{R}}(g_{A})$. In particular, a point in $\mathcal{R}$
		belongs to $\Lambda_{k}(A)$ if and only if its belongs to the interior of
		$\Lambda_{k}(A)$. If there is any such point, then $\Lambda_{k}(A)$ is two-dimensional.
		Conversely, if $\Lambda_{k}(A)$ is two-dimensional, then it will contain some
		non-empty bounded connected component of $\C\backslash\{a+\ii b : g(a,b)=0\}$.
		The representative of this connected component gives a point in
		$\mathcal{R}\cap \Lambda_{k}(A)$. Moreover, by applying the membership test
		to $\mathcal{R}$, a set of representatives of each bounded connected component
		of the complement of $g_{A}$, we obtain at least one point for each
		connected component of $\Lambda_{k}(A)\backslash\{a+\ii b:g_{A}(a,b)=0\}$.

		If $\Lambda_{k}(A) = \{a+\ii b\}$, then, by \autoref{thm:tritangent}, we conclude
		that $t+ax+by$ is either a tritangent or a linear factor of $f_{A}$. In
		particular, $a+\ii b\in \mathcal{C}\cap \Lambda_{k}(A)$.

		Finally, if $\Lambda_{k}(A)$ is empty, then no points in $\mathcal{T}$,
		$\mathcal{R}$ or $\mathcal{C}$ past the membership test in
		\Cref{alg:membership} and we correctly conclude that
		$\dim(\Lambda_{k}(A))=-1$.
	\end{proof}

	We finish the discussion of the previous algorithm by describing how to
	compute the sets $\mathcal{R}$ and $\mathcal{C}$. We give an algorithm for computing
	$\mathcal{R}$ based on the theory of cylindrical algebraic decompositions for semialgebraic
	sets. For a thorough presentation see \cite[Section 5.1]{Basu2006}.

	\begin{algorithm}
		[H]
		\caption{Bounded connected components of $\R^{2} \backslash \mathcal{V}(g)$
		\smallskip
		\\
		\textbf{Input:} $g(a,b)$ a square-free polynomial in $\R[a,b]$ \\
		\textbf{Output:} A set of representatives for each bounded connected
		component of $\R^{2} \backslash \mathcal{V}(g)$.
		\smallskip
		}
		\label{alg:connectedComp}
		\begin{algorithmic}
			[0] \State Set $h ={\rm Disc}_{a}(g)$. \State Compute the distinct real
			roots $b_{1} < \ldots < b_{m}$ of $h$. \For{$i = 1, \hdots, m-1$} \State
			define $b'_{i} = (b_{i} + b_{i+1})/2$. \EndFor \For{$i = 1, \hdots, m-1$}
			\State Find the distinct real roots $a_{i1}< \ldots < a_{in_i }$ of
			$g(a,b_{i}')$. \For{$j = 1, \hdots, n_{i}$} \State define
			$a'_{ij}= (a_{ij}+ a_{i(j+1)})/2$. \EndFor \EndFor \State Set $\mathcal{R}=
			\bigcup_{i=1}^{m-1}\{ a'_{ij}+\ii b'_{i},\ldots, a'_{i(n_i-1)}+ \ii b'_{i}
			\}$. \State \Return{$\mathcal{R}$.}
		\end{algorithmic}\label{alg:connectedComponents}
	\end{algorithm}

	\begin{remark}[Computing tritangent lines and linear factors]
		\label{rem:tritangents} Let $f$ be homogeneous polynomial $f\in \C[t,x,y]$ and
		let $f^{\rm red}$ be the square-free part of $f$, which is homogeneous of
		some degree $d$. We can compute the set $(a,b)$ for which the line $t+ax+by$
		is tritangent to $\V(f)$ or divides $f$ via elimination theory. We note that
		when $t+ax+by$ divides $f$, the restriction of $f$ to this line is identically
		zero, which satisfies our definition for being tritangent, by taking $h=0$.
		Note that if $d< 6$, then the only lines tritangent to $f$ are its linear factors.
		The set of $(a,b)$ for which $t+ax+by$ divides $f^{\rm red}$ is the variety
		of the ideal generated by the $d+1$ coefficients of
		$f^{\rm red}(-ax-by,x,y)$ as a polynomial in $x$ and $y$.

		For $d\geq 6$, consider the ideal $I \subset \C[a,b,\ell_{00},\ell_{01},\ell_{10}
		,\ell_{11},\ell_{20},\ell_{21}, h_{0}, \hdots, h_{d-6}]$ generated by the $d+
		1$ coefficients of
		\begin{equation}
			\label{eq:tritangentCondition}f^{\rm red}(-ax-by,x,y) -(\ell_{00}x+\ell_{01}
			y)^{2}\cdot (\ell_{10}x+\ell_{11}y)^{2}\cdot(\ell_{11}x+\ell_{12}y)^{2}\cdot
			\left(\sum_{j=0}^{d-6}h_{j}x^{j}y^{d-6}\right)
		\end{equation}
		in $x$ and $y$. The line $t+ax+by$ is tritangent to $\V(f)$ if and only if
		there exists values of $\ell_{ij}$ and $h_{j}$ for which the polynomial in
		\eqref{eq:tritangentCondition} is identically zero. The set of $(a,b)$ for
		which this holds is therefore equal to the variety of the elimination ideal
		\[
			J ={\rm radical}\left({\rm eliminate}\left(I, \{\ell_{00},\ell_{01},\ell_{10}
			,\ell_{11},\ell_{20},\ell_{21}, h_{0}, \hdots, h_{d-6}\}\right)\right).
		\]
	\end{remark}

	\begin{remark}
		For the above algorithms we focus on correctness rather than computational efficiency.
		For an implementation of the above algorithm some modifications can be made
		in order to make it more computationally efficient. It is also worth noting that
		most of the sets of points that we need can be reused when computing $\Lambda
		_{k}(A)$ for different $k'$s. The polynomial $g_{A}$ is also independent of $k$
		and therefore it needs to be computed only once.
	\end{remark}

	\section{Gallery of Examples}
	\label{sec:gallery}

	In this section we illustrate the behavior of the given algorithms of
	\autoref{sec:algorithm} and highlight the many possible subtle behaviors of
	higher rank numerical ranges.

	\begin{example}
		\label{ex:tritangent} In this example, we construct a matrix $\tilde{A}$
		with $0$-dimensional $\Lambda_{k}(\tilde{A})$ coming from a tritangent line
		not passing through any singularities of $\V(f_{\tilde{A}})$.

		To start, let $A$ denote the matrix from Example~\ref{ex:quartic1}. The
		polynomial $f_{A}$ defines a quartic plane curve. We define a parametric family
		of $6\times 6$ matrices by
		\[
			\tilde{A}(u) =
			\begin{pmatrix}
				B + u(1+\ii)I_{2} & 0 \\
				0                 & A
			\end{pmatrix}
			\ \ \ \text{ where }\ \ \ B =
			\begin{pmatrix}
				0 & 3 \\
				0 & 0 \\
			\end{pmatrix}
		\]
		and $I_{2}$ denotes the $2\times 2$ identity matrix. The boundary of the
		numerical range of $B$ is the circle $\{a + \ii b : 4a^{2} + 4b^{2} = 9\}$. The
		boundary of the numerical range of $B + u(1+\ii)I_{2}$ is this circle translated
		by $u(1+\ii)$. \autoref{fig:quarticxHyperbola} shows the Kippenhahn curves
		and numerical ranges of $\tilde{A}(u)$ for $u=-1, \hat{u}$ and $-9/4$ where
		\[
			\hat{u}=\frac{-4 \sqrt{540+330 \sqrt{3}}-3 \sqrt{1374+792 \sqrt{3}}}{72+44
			\sqrt{3}}.
		\]
		We see that for $u=-1$, $\Lambda_{3}(\tilde{A})$ is full dimensional,
		corresponding to the two-dimensional family of lines not intersecting the
		curve $\V_{\R}(f_{\tilde{A}})$. On the other hand, for $u=-9/4$, $\Lambda_{3}
		(\tilde{A})$ is empty, as every line intersects regions of
		$\R^{3}\backslash \V_{\R}(f_{\tilde{A}})$ at which the matrix
		$tI + x\Re(\tilde{A}) + y\Im(\tilde{A})$ has four positive or negative
		eigenvalues. The value of $\hat{u}$ was computed so that for $u=\hat{u}$,
		$\Lambda_{3}(\tilde{A})$ is a single point. This corresponds to the unique line
		on which the matrix pencil has at most three positive and at most three negative
		eigenvalues. This line is tritangent to the curve $\V_{\R}(f_{\tilde{A}})$, as
		promised by \autoref{thm:tritangent}.

		In \autoref{fig:quarticxHyperbola}, the sets $\Lambda_{2}({\tilde{A}})$ are
		shown in light blue. As promised by \autoref{thm:boundary}, the boundary of
		these regions is contained in the zero locus of $g_{\tilde{A}}$.
	\end{example}

	\begin{figure}
		\centering
		\includegraphics[height=2in]{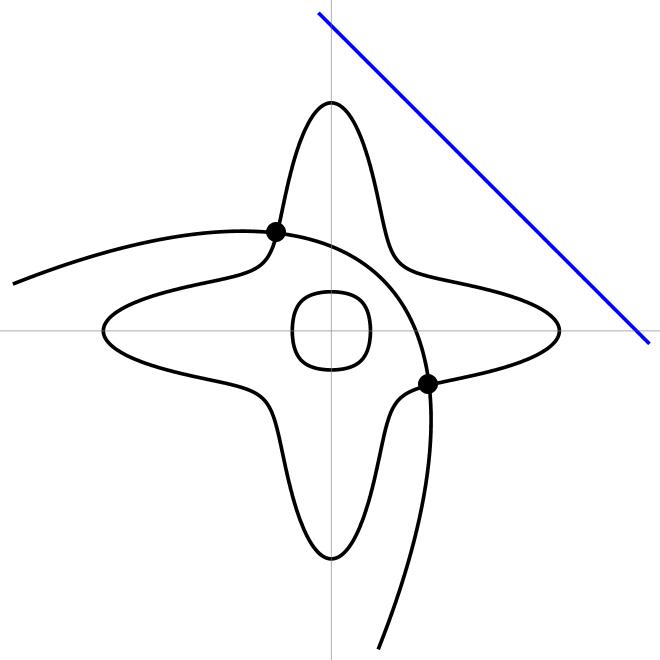}
		\centering
		\includegraphics[height=2in]{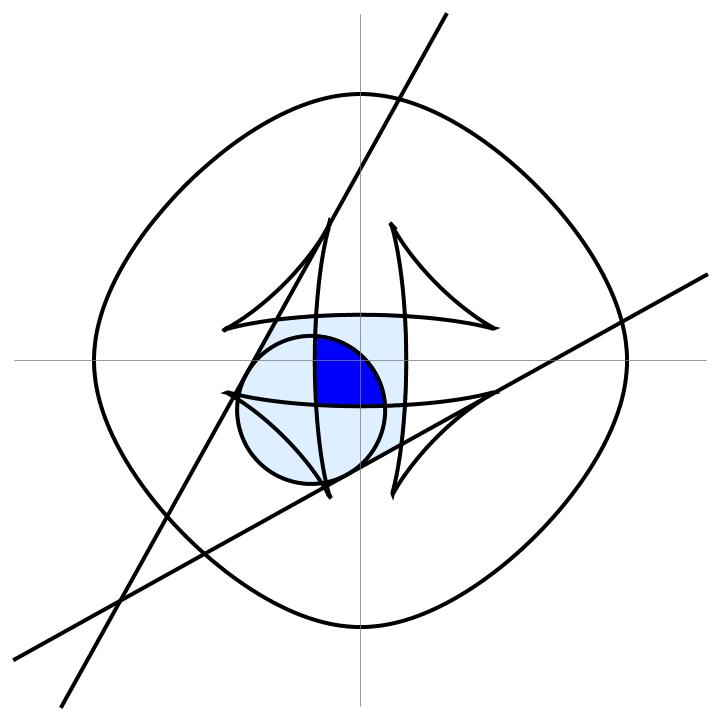}
		\vfill
		\includegraphics[height=2in]{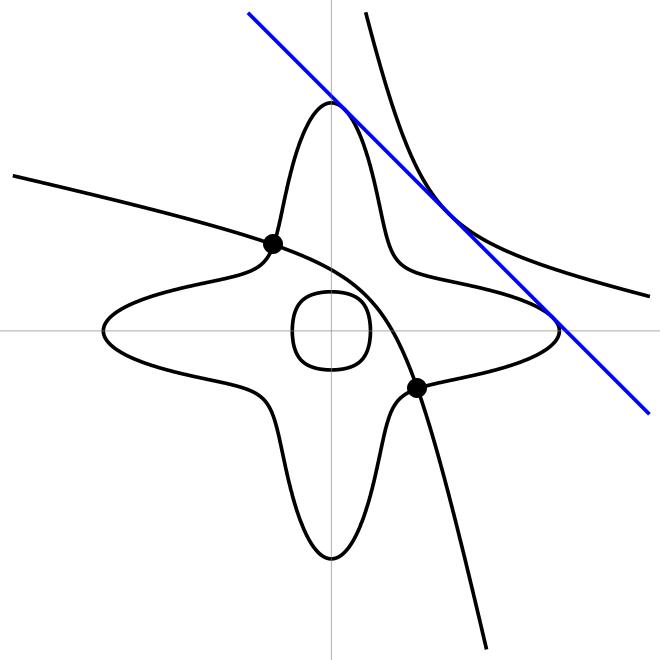}
		\centering
		\includegraphics[height=2in]{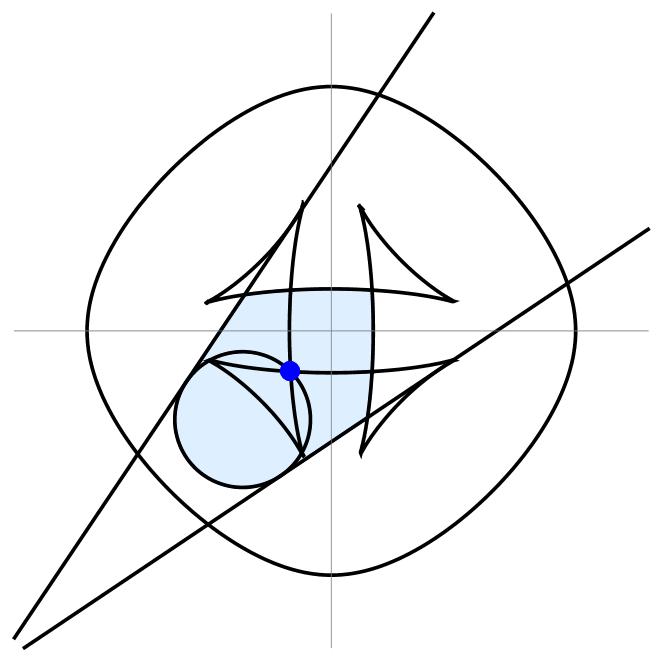}
		\vfill
		\includegraphics[height=2in]{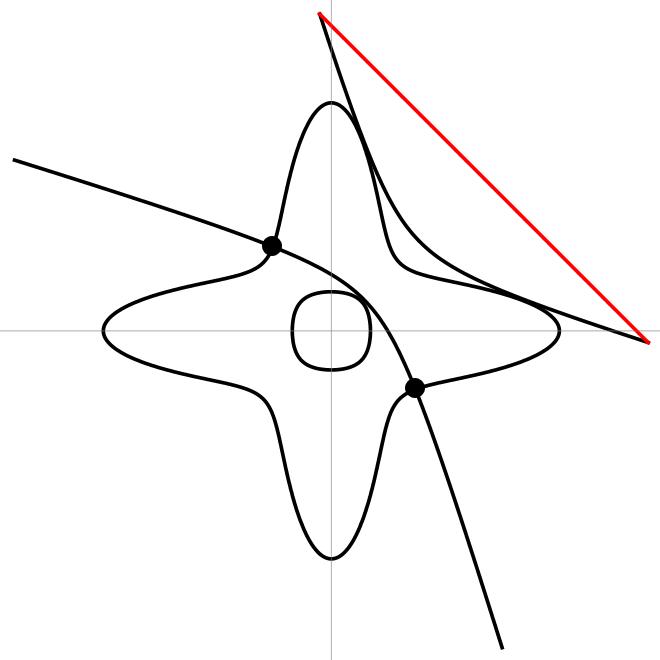}
		\centering
		\includegraphics[height=2in]{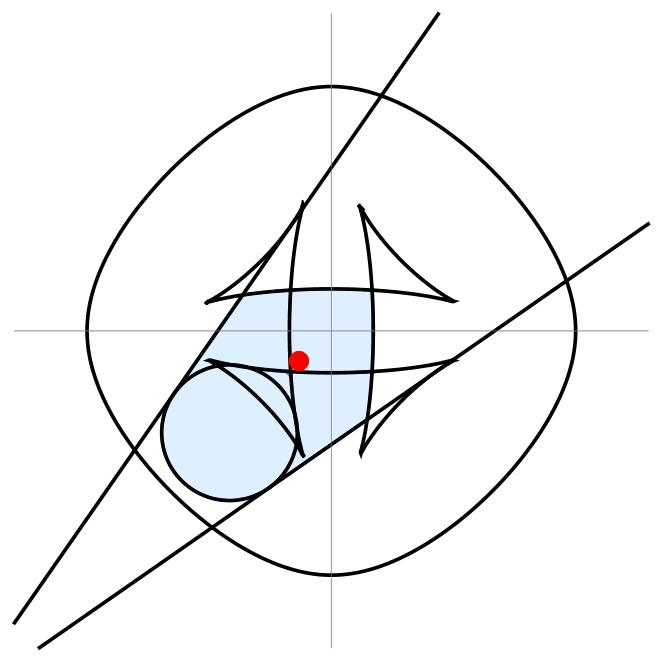}

		\caption{Kippenhahn curve and dual curve for Example~\ref{ex:tritangent}. Here
		we study a family of curves obtained by allowing the circle on the right to be
		shifted along the line $a = b$. The regions shaded by light and dark blue
		correspond to $\Lambda_{2}(\tilde{A})$ and $\Lambda_{3}(\tilde{A})$, respectively.
		As we translate the circle, the rank-3 numerical ranges passes from being 2-dimensional
		to 0-dimensional to empty.}
		\label{fig:quarticxHyperbola}
	\end{figure}

	\begin{figure}
		\includegraphics[height=2in]{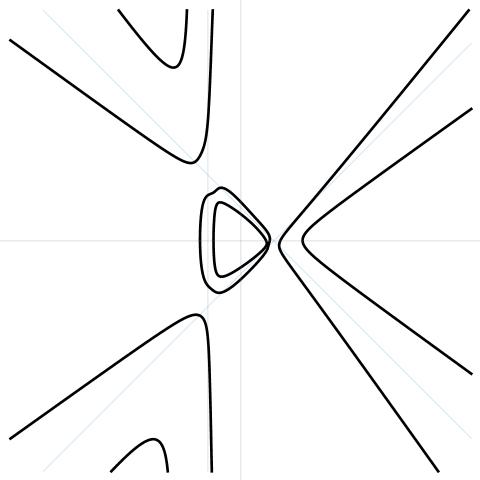}
		\includegraphics[height=2in]{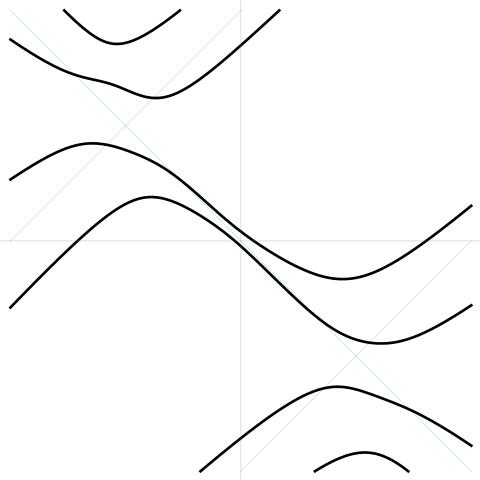}
		\caption{Kippenhahn curve for Example~\ref{ex:SmoothSextic} in the charts
		$t = 1$ and $y = 1$. }
		\label{fig:sextic}
	\end{figure}

	\begin{example}
		\label{ex:SmoothSextic} Here we give a $6\times 6$ matrix $\tilde{A}$ for which
		the curve $\V(f_{\tilde{A}})$ is irreducible and smooth but the numerical range
		$\Lambda_{3}(\tilde{A})$ is empty.

		Let $A$ denote the diagonal matrix ${\rm diag}(1,1, -1 + \ii, -1 + \ii, -1 -\ii
		, -1 -\ii )$. By the formula of $\Lambda_{k}(A)$ for normal matrices we know
		that $\Lambda_{3}(A)$ is empty. To construct a smooth curve with the same property
		we do a small perturbation. Specifically, define
		\[
			B =
			\begin{pmatrix}
				0  & 1  & 1 & -1 & 0 & 1  \\
				1  & 0  & 0 & 0  & 1 & 0  \\
				-1 & 1  & 1 & -1 & 1 & 1  \\
				1  & 1  & 1 & 0  & 1 & 0  \\
				0  & -1 & 1 & 1  & 1 & -1 \\
				-1 & 0  & 1 & 0  & 1 & -1
			\end{pmatrix}.
		\]
		The matrix $\tilde{A}= A + (1/5)B$ has the desired properties. The curve
		$\V(f_{\tilde{A}})$ is shown in the affine charts $t=1$ and $y=1$ in
		\autoref{fig:sextic}.
	\end{example}

	\begin{example}
		\label{ex:pringle2} As we saw in \autoref{sec:Dim01}, zero-dimensional rank-$k$
		numerical ranges can also come from antipodal points on the curve $O_{k}(A)$,
		which give rise to singularities in the Kippenhahn curve $\V_{\R}(f_{A})$.
		Consider the $4\times 4$ matrix $A$ from Example~\ref{ex:pringle} with
		Kippenhahn polynomial
		$f_{A} = t^{4} - 5 t^{2} x^{2} + 4 x^{4} - t^{2} y^{2}.$ The point $0+0\ii$,
		corresponding to the line $t = 0$, is the only one that passes the
		membership test for $\Lambda_{2}(A)$. The curve $\V_{\R}(f_{A})$ and line $t=
		0$ in the affine chart $y=1$ are shown in \autoref{fig:pringle2}.
	\end{example}

	\begin{figure}
		\includegraphics[height=2in]{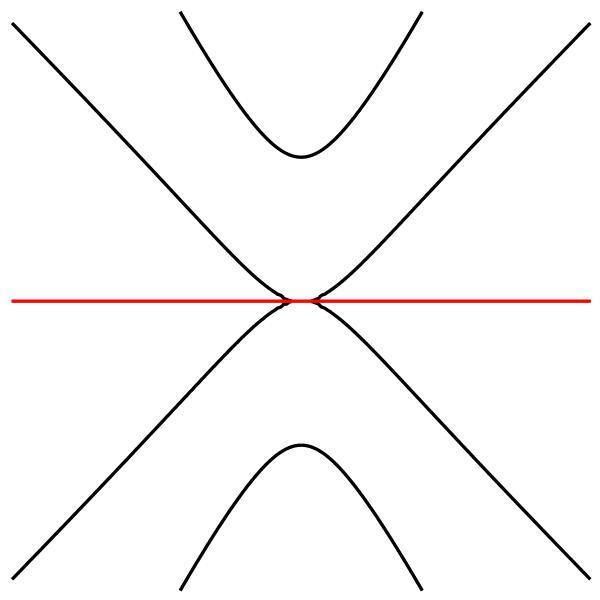}
		\caption{The Kippenhahn curve $\V_{\R}(f_{A})$ from Example~\ref{ex:pringle2}.}
		\label{fig:pringle2}
	\end{figure}

	If $a+\ii b$ is the endpoint of a one-dimensional $\Lambda_{k}(A)$, the line $t
	+ax+by$ is either tangent to $\V_{\R}(f_{A})$ at multiple points or has a zero
	of multiplicity $d + 1$ at a singularity of multiplicity $d$. One might think
	that the algebraic condition for $\Lambda_{k}(A)$ to be a single point should
	be a degeneration of the algebraic condition for two endpoints of a one-dimesional
	$\Lambda_{k}(A)$. This would suggest that the corresponding line is either tangent
	to the curve $\V_{\R}(f_{A})$ at multiple points or has a zero of multiplicity
	$d + 2$ at a singularity of multiplicity $d$. The next example shows that this
	is not the case.

	\begin{example}
		\label{ex:weirdTritangent} Consider the matrix
		\[
			A =
			\begin{pmatrix}
				1 + \ii & 1  & 1  & 0  \\
				1       & -1 & -1 & 0  \\
				1       & -1 & -1 & -1 \\
				0       & 0  & -1 & 0
			\end{pmatrix}.
		\]
		The Kippenhahn polynomial is
		\[
			f_{A}(t,x,y) = t^{4}-t^{3} x+t^{3} y-5 t^{2} x^{2}-2 t^{2} x y-t x^{2} y+2
			x^{4}+x^{3} y
		\]
		and its dual curve is given by
		\begin{align*}
			 & g_{A}= 49 a^{6}-644 a^{5} b+196 a^{5}+3824 a^{4} b^{2}-2212 a^{4} b+98 a^{4}-12172 a^{3} b^{3}+8942 a^{3}b^{2}+56 a^{3} b- \\
			 & 294 a^{3}+21248 a^{2} b^{4}-16084 a^{2} b^{3}-3420 a^{2} b^{2}+2324 a^{2} b-147 a^{2}-18836 a b^{5}+11860 a b^{4}+         \\
			 & 9244 a b^{3}-4754 a b^{2}+56 a b+98 a+7260 b^{6}-5132 b^{5}-1739 b^{4}-1600 b^{3}+2402 b^{2} -672 b+49.
		\end{align*}
		The curve $\V_{\R}(f_{A})$ has a singularity at $p=[0:0:1]$ and $\lambda_{2}(
		0,1) = \lambda_{3}(0,1) =\lambda_{4}(0,1) = 0$. By \autoref{thm:ptangents}, it
		follows that $\Lambda_{2}(A)$ is contained in the line $b = 0$. The curve
		has three $p$-tangent lines. The lines and the corresponding points of $\V_{\R}
		(g_{A})$ are shown in \autoref{fig:singleton_ptangents}. Of these, only one point
		passes the membership test for $\Lambda_{2}(A)$ and we conclude that $\Lambda
		_{2}(A)$ is this single point.

		The curve $\V(f_{A})$ has multiplicity three at $p$. The restriction of
		$f_{A}$ to each $p$-tangent line has a single root of multiplicity four. In particular,
		this shows that we are unable to use algebraic methods to distinguish
		between the point of a zero-dimensional set $\Lambda_{k}(A)$ and the
		endpoint of a one-dimensional set $\Lambda_{k}(A)$.

		The curve $\V_{\R}(f_{A})$ in the affine chart $y=1$ is the union of the
		graphs of the function $t = -\lambda_{k}(x\Re(A) + \Im(A))$ for $k=1,2,3,4$.
		In many other examples, the branches of $\V_{\R}(f_{A})$ meeting at a
		singularity with a common tangent line correspond to the $\lambda_{k}$ and
		$\lambda_{n-k}$ eigenfunctions. Here the branches $\lambda_{2}$ and $\lambda_{4}$
		meet along common tangent lines at $p$.
	\end{example}
	\begin{figure}
		\includegraphics[height=2in]{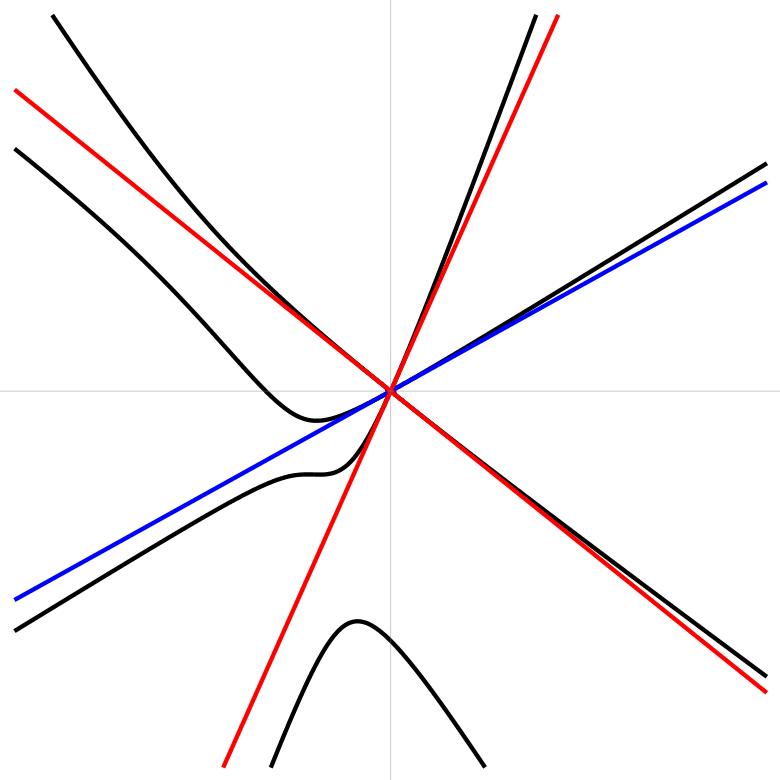}
		\includegraphics[height=2in]{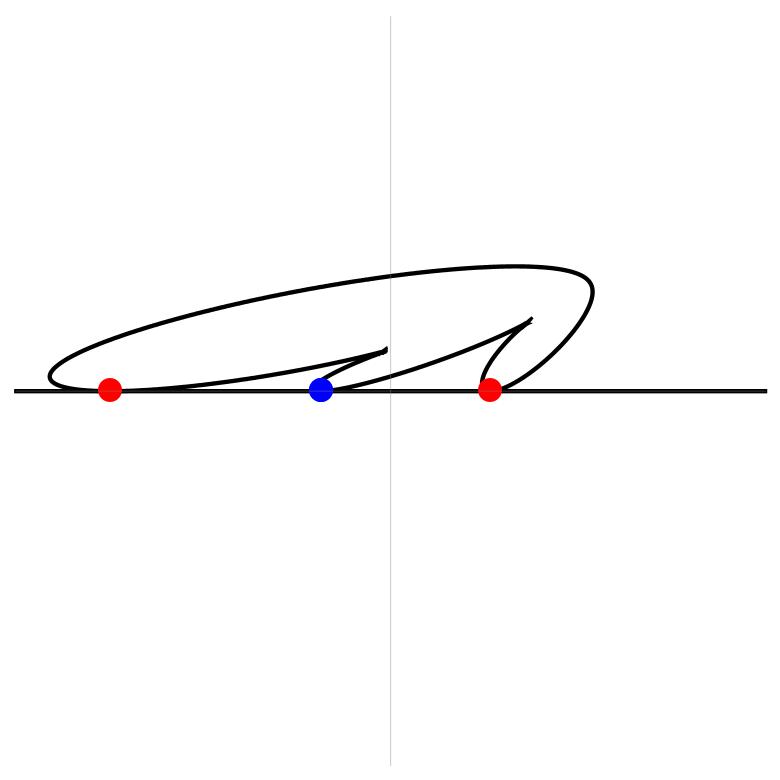}
		\caption{The curve $\V(f_{A})$ from Example~\ref{ex:weirdTritangent} along
		with its three $[0:0:1]$-tangent lines. On the right is the curve $\V_{\R}(g_{A}
		)$ along with the points corresponding to these lines. The unique point in
		$\Lambda_{2}(A)$ and corresponding line are shown in blue. }
		\label{fig:singleton_ptangents}
	\end{figure}

	\bibliographystyle{plain}
	\bibliography{references}
\end{document}